\documentclass[11pt]{article}

\usepackage{amsmath, amssymb}
\usepackage{authblk}
\usepackage{geometry}
\usepackage{hyperref}
\usepackage{graphicx}
\usepackage{epstopdf}
\usepackage{mathrsfs}
\usepackage{amsfonts}
\usepackage{amscd,amsthm,bm,psfrag}
\usepackage[numbers,sort&compress]{natbib}

\makeatletter
\renewcommand{\@seccntformat}[1]{{\csname the#1\endcsname}{\normalsize.}\hspace{.5em}}

\makeatother

\numberwithin{equation}{section}

\def \[{\begin{equation}}
\def \]{\end{equation}}

\newtheorem{thm}{Theorem}[section]
\newtheorem{defi}[thm]{Definition}
\newtheorem{lem}[thm]{Lemma}
\newtheorem{cor}[thm]{Corollary}
\newtheorem{prop}[thm]{Proposition}

\newtheorem{con}[thm]{Conjecture}

\newtheorem*{thm*}{Theorem}
\newtheorem*{prop*}{Proposition}

\newtheorem{clm}[thm]{Claim}

\newenvironment{thmbis}[1]
{\renewcommand{\thethm}{\ref{#1}}\addtocounter{thm}{-1}
\begin{thm}}
{\end{thm}}

\newenvironment{propbis}[1]
{\renewcommand{\thethm}{\ref{#1}}\addtocounter{thm}{-1}
\begin{prop}}
{\end{prop}}

\newcommand\ex{\ensuremath{\mathrm{ex}}}

\newcommand\cC{{\mathcal C}}

\newcommand\cF{{\mathcal F}}
\newcommand\cG{{\mathcal G}}
\newcommand\cH{{\mathcal H}}

\newcommand\cK{{\mathcal K}}
\newcommand\cL{{\mathcal L}}

\newcommand\cN{{\mathcal N}}

\newcommand\cS{{\mathcal S}}
\newcommand\cT{{\mathcal T}}

\begin{document}
\title{
{On generalized Tur\'{a}n problems for expansions} 
}

\author[a,b]{Junpeng Zhou\,}
\author[c]{Xiamiao Zhao\,}
\author[a,b,*]{Xiying Yuan\,}

\affil[a]{\small \,Department of Mathematics, Shanghai University, Shanghai 200444, PR China}
\affil[b]{\small \,Newtouch Center for Mathematics of Shanghai University, Shanghai 200444, PR China}
\affil[c]{\small \,Department of Mathematical Sciences, Tsinghua University, Beijing 100084, PR China}

\date{}

\maketitle

\footnotetext{*\textit{Corresponding author}.}
\footnotetext{
Email addresses: \texttt{junpengzhou@shu.edu.cn} (J. Zhou), 
\texttt{zxm23@mails.tsinghua.edu.cn} (X. Zhao)
\texttt{xiyingyuan@shu.edu.cn} (X. Yuan).}





\begin{abstract}
Given a graph $F$, the $r$-expansion $F^r$ of $F$ is the $r$-uniform hypergraph obtained from $F$ by inserting $r-2$ new distinct vertices in each edge of $F$.
Given $r$-uniform hypergraphs $\cH$ and $\cF$, the generalized Tur\'{a}n number, denoted by $\ex_r(n,\cH,\cF)$, is the maximum number of copies of $\cH$ in an $n$-vertex $r$-uniform hypergraph that does not contain $\cF$ as a subhypergraph.
In the case where $r=2$ (i.e., the graph case), the study of generalized Tur\'{a}n problems was initiated by Alon and Shikhelman [\textit{J. Combin. Theory Series B.} 121 (2016) 146--172].
Motivated by their work, we systematically study generalized Tur\'{a}n problems for expansions and obtain several general and exact results.
In particular, for the non-degenerate case, we determine the exact generalized Tur\'{a}n number for expansions of complete graphs, and establish the asymptotics of the generalized Tur\'{a}n number for expansions of the vertex-disjoint union of complete graphs. For the degenerate case, we establish the asymptotics of generalized Tur\'{a}n numbers for expansions of several classes of forests, including star forests, linear forests and star-path forests.
\end{abstract}




\section{\normalsize Introduction}
An \textit{$r$-uniform hypergraph} ($r$-graph for short) $\cH=(V(\cH),E(\cH))$ consists of a vertex set $V(\cH)$ and a hyperedge set $E(\cH)$, where each hyperedge in $E(\cH)$ is an $r$-subset of $V(\cH)$. The size of $E(\cH)$ is denoted by $e(\cH)$.
Let $\cH_1{\cup} \cH_2$ denote the vertex-disjoint union of $r$-graphs $\cH_1$ and $\cH_2$, and let $k\cH$ denote the vertex-disjoint union of $k$ $r$-graphs $\cH$.

Let $\mathcal{F}$ be an $r$-graph. An $r$-graph $\cH$ is \textit{$\mathcal{F}$-free} if $\cH$ does not contain $\mathcal{F}$ as a subhypergraph. The \textit{Tur\'{a}n number} of $\mathcal{F}$, denoted by ${\rm{ex}}_r(n,\mathcal{F})$, is the maximum number of hyperedges in an $n$-vertex $\mathcal{F}$-free $r$-graph.
When $r=2$, we use $\mathrm{ex}(n,\mathcal{F})$ instead of $\mathrm{ex}_2(n,\mathcal{F})$, which represents a fundamental and well-studied problem in extremal graph theory.
A classical result in extremal graph theory is Tur\'{a}n theorem \cite{Tu}, which determines the exact Tur\'{a}n number for the $\ell$-vertex complete graph $K_\ell$. The Erd\H{o}s-Stone-Simonovits theorem \cite{ESi,ESt} gives an asymptotics of the Tur\'{a}n number for any $k$-chromatic graph.
When $\cF$ is bipartite, the problem of determining $\ex(n,\cF)$ remains an active topic in extremal graph theory. For an extensive overview of the historical development, we refer the reader to the survey by F\"{u}redi and Simonovits \cite{FuSi}. In particular, Erd\H os and S\'os \cite{Er} conjectured ${\rm ex}(n,\cF)\leq \frac{n(\ell-1)}{2}$ when $\cF$ is a tree with $\ell$ edges. The conjecture is known to hold for paths by the Erd\H os-Gallai theorem \cite{EGa} and for spiders \cite{FHL}. Unlike the conjecture proposed for trees, the Tur\'{a}n number of a forest is highly dependent on its structural characteristics \cite{KP}.
Let $M_{s+1}$ denote a matching of size $s+1$, i.e., the graph consisting of $s+1$ independent edges. In 1959, Erd\H{o}s and Gallai \cite{EGa} determined the Tur\'{a}n number of $M_{s+1}$.
Let $P_\ell$ and $S_\ell$ denote a path and a star with $\ell$ edges, respectively. A \textit{star forest} is a forest whose connected components are stars and a \textit{linear forest} is a forest whose connected components are paths. Bushaw and Kettle \cite{BuKe} determined the Tur\'{a}n number of $kP_\ell$ for $k \geq 2$, $\ell \geq 3$ and sufficiently large $n$. Later, Lidick\'{y}, Liu and Palmer \cite{LiLP} generalized the result of Bushaw and Kettle \cite{BuKe}, and simultaneously determined the Tur\'{a}n number of star forests.

It is well known that hypergraph Tur\'{a}n problems are considerably more difficult than in the graph setting. Some natural questions are surprisingly difficult. For example, the Tur\'{a}n number and even the asymptotics for $K_4^{(3)}$, the 4-vertex complete $3$-graph, remains unknown. 
The open problem shows that even for hypergraphs with few vertices and hyperedges, the Tur\'{a}n problem may be complicated. A fundamental object of investigation in hypergraph Tur\'{a}n problems is expansions.

Given a graph $F$, the \textit{$r$-uniform expansion} (or briefly $r$-expansion) $F^r$ of $F$ is the $r$-graph obtained from $F$ by inserting $r-2$ new distinct vertices in each edge of $F$, such that the $(r-2)e(F)$ new vertices are distinct from each other and are not in $V(F)$. It was introduced by Mubayi in \cite{Mu}. Tur\'{a}n problems on expansions have been extensively studied, yielding numerous related extremal results.
For expansions of complete graphs, Erd\H{o}s \cite{Er2} conjectured that ${\rm{ex}}_r(n,K_3^r)=\binom{n - 1}{r - 1}$ for $n\geq \frac{3}{2}r$. The conjecture was later proved by Mubayi and Verstra\"{e}te \cite{MuV}. Let $\cT_r(n,\ell)$ denote the complete balanced $\ell$-partite $r$-graph. For convenience, let $t_r(n,\ell):=e(\cT_r(n,\ell))$. Mubayi \cite{Mu} conjectured that $\cT_r(n,\ell)$ is the unique maximum $K_{\ell+1}^r$-free $r$-graph for sufficiently large $n$. The conjecture was later proved by Pikhurko \cite{Pi}.

\begin{thm}[Pikhurko \cite{Pi}]\label{Pikhurko}
Let integers $\ell \geq r\geq 3$. If $n$ is sufficiently large, then $\ex_r(n, K_{\ell+1}^{r}) =t_r(n,\ell)$ and $\cT_r(n,\ell)$ is the unique extremal hypergraph.
\end{thm}

Gerbner \cite{Ge2} determined the exact Tur\'{a}n number for expansions of the vertex-disjoint union of complete graphs. Let $S_{n,t}^r(n-t,\ell)$ denote the following $r$-graph. We take a set $U$ of $t$ vertices and a copy of $\cT_r(n-t,\ell)$, then add each $r$-set containing at least one vertex from $U$ as a hyperedge.

\begin{thm}[Gerbner \cite{Ge2}]\label{union-complete-graphs}
Let $F$ consist of $k$ components with chromatic number $\ell+1$, each with a color-critical edge, and any number of components with chromatic number at most $\ell$. Then
${\rm{ex}}_r(n,F^r)=e(S_{n,k-1}^r(n-k+1,\ell))$ for sufficiently large $n$.
\end{thm}

For expansions of stars, Duke and Erd\H{o}s \cite{AA4} gave an upper bound for ${\rm{ex}}_r(n,S_\ell^r)$ as follows.

\begin{thm}[Duke and Erd\H{o}s \cite{AA4}]\label{lem3.2}
Fix integers $r\geq3$ and $\ell\geq2$. Then there exists a constant $c(r)$ such that for sufficiently large $n$,
\begin{eqnarray*}
{\rm{ex}}_r(n,S_\ell^r)\leq c(r)\ell(\ell-1)n^{r-2}.
\end{eqnarray*}
\end{thm}

For expansions of paths, F\"{u}redi, Jiang and Seiver \cite{FuJS} determined ${\rm{ex}}_r(n,P_\ell^r)$ for $r,\ell\geq4$ and sufficiently large $n$. The case when $r\geq3$ and $\ell\geq4$ was proved by Kostochka, Mubayi and Verstra\"{e}te \cite{KMV} using methods that differ significantly from those in \cite{FuJS}. The case $\ell=3$ was settled by Frankl and F\"{u}redi \cite{FrFu}. The case when $\ell=2$ and $r\geq4$ was settled by Frankl \cite{Fran} and the case when $\ell=2$, $r=3$ and all $n$ was settled by S\'{o}s \cite{Sos}. Furthermore, Kostochka, Mubayi and Verstra\"{e}te \cite{KMV} also determined ${\rm{ex}}_r(n,C_\ell^r)$ for all $r\geq3$, $\ell\geq4$ and sufficiently large $n$, where $C_\ell$ denotes a cycle with $\ell$ edges.


Now let us consider expansions of forests. In 1965, Erd\H{o}s \cite{Er1} proposed the following conjecture for matchings.

\begin{con}[Erd\H{o}s Matching Conjecture~\cite{Er1}]\label{con1}
Let integers $r\geq2$, $s\geq1$ and $n\geq (s+1)r-1$. Then $$\ex_r(n,M_{s+1}^r)\leq \max\left\{ \binom{(s+1)r-1}{r}, \binom{n}{r} - \binom{n-s}{r} \right\}.$$
\end{con}
It is worth noting that the conjecture is known to hold for the cases $r=2$ \cite{EGa} and $r=3$ \cite{Fran2017,LuMi}. For the general case, Erd\H{o}s \cite{Er1} verified Conjecture \ref{con1} for large $n$. Later, the threshold was improved in several papers by Bollob\'as, Daykin, Erd\H{o}s \cite{68}, Huang, Loh, Sudakov \cite{312}, Frankl, \L uczak, Mieczkowska \cite{215}. For further developments and relevant results concerning this conjecture, one can refer to \cite{Fr1,FrKu,213,KoKu}.
Very recently, Kupavskii and Sokolov \cite{KuSo} completely resolved the conjecture for the case $n\leq 3(s+1)$.


Bushaw and Kettle \cite{BK} determined the exact Tur\'{a}n number for expansions of linear forests when $n$ is sufficiently large, as follows. 

\begin{thm}[Bushaw and Kettle \cite{BK}]\label{Bushaw-Kettle}
Let $r \geq 3$, $k \geq 2$, $\ell_1, \ldots, \ell_k \geq 1$, and $n$ be sufficiently large. Setting $t = \sum_{i=1}^k\lfloor\frac{\ell_i + 1}{2}\rfloor - 1$ and $d_\ell = 0$ when at least one of the $\ell_i$ is odd and $d_\ell = \dbinom{n - t - 2}{r - 2}$ otherwise, the following holds:
$$\mathrm{ex}_r(n,P_{\ell_1}^{r}\cup\cdots\cup P_{\ell_k}^{r}) = \dbinom{n - 1}{r - 1} + \cdots + \dbinom{n - t}{r - 1} + d_\ell.$$
\end{thm}

Khormali and Palmer \cite{KP} 
obtained the following result for expansions of star forests. 

\begin{thm}[Khormali and Palmer \cite{KP}]\label{Khormali-Palmer}
Fix integers $\ell,k\geq1$ and $r\geq2$. Then for sufficiently large $n$,
$${\rm{ex}}_r(n,kS_\ell^r)=\binom{n}{r}-\binom{n-k+1}{r}+{\rm{ex}}_r(n-k+1,S_\ell^r).$$
\end{thm}

Recently, Zhou and Yuan \cite{ZhY,ZhY2} studied the Tur\'{a}n problem on expansions of star-path forests, where a \textit{star-path forest} is a forest whose connected components are either paths or stars. For further results on Tur\'{a}n problems for expansions, one can refer to Subsection 5.2.1 of \cite{A666}, as well as the survey  \cite{MuV2} by Mubayi and Verstra\"{e}te.

Given two graphs $H$ and $F$, let $\ex(n,H,F)$ denote the maximum number of copies of $H$ in an $n$-vertex $F$-free graph. Alon and Shikhelman \cite{AlSh} initiated the systematic study of the function $\ex(n,H,F)$, which is often called the \textit{generalized Tur\'{a}n problem}. In the case where $H=K_s$, 
Zykov \cite{Zyko} and independently Erd\H{o}s \cite{Erdo} determined the exact value of $\ex(n,K_s,K_\ell)$. Wang \cite{Wang} determined the exact value of $\ex(n,K_s,M_{k+1})$. In \cite{ChCh}, Chakraborti and Chen determined the exact value of $\ex(n,K_s,P_\ell)$. Given $r$-graphs $\cH$ and $\cG$, let $\mathcal{N}(\cH,\cG)$ denote the number of copies of $\cH$ in $\cG$.

\begin{thm}[Chakraborti and Chen \cite{ChCh}]\label{CC}
For any positive integers $n$ and $3\leq s\leq\ell$, we have $\ex(n,K_s,P_\ell)=\cN(K_s,pK_{\ell}\cup K_q)$, where $n=p\ell+q$ and $0\leq q\leq \ell-1$. 
\end{thm}

Later, Zhu and Chen \cite{ZhCh} determined $\ex(n,K_s,F)$ when $F$ is the vertex-disjoint union of paths of length at least $4$ and $n$ is sufficiently large. Khormali and Palmer \cite{KP} determined $\text{ex}(n,K_s,kS_\ell)$ if $3\leq s\leq\ell+k-2$, $\ell$ divides $n-k+1$ and $n$ is sufficiently large. For a very recent survey on generalized Tur\'{a}n problems, one can refer to the work of Gerbner and Palmer \cite{GePa}.

It is natural to investigate the generalized Tur\'{a}n problem on hypergraphs. However, determining generalized Tur\'{a}n numbers for hypergraphs is more challenging than determining their classical Tur\'{a}n number counterparts. Nevertheless, several sporadic results on generalized Tur\'{a}n problems for hypergraphs have still been established. 
Before stating these results, we introduce several notations. Given $r$-graphs $\cH$ and $\cF$, let $\ex_r(n,\cH,\cF)$ denote the maximum number of copies of $\cH$ in an $n$-vertex $\cF$-free $r$-graph. When $r=2$, we use $\ex(n,\cH,\cF)$ instead of $\ex_2(n,\cH,\cF)$. We denote by $K_s^{(r)}$ the complete $r$-graph on $s$ vertices if $s\geq r$, and define $K_s^{(r)}$ as an $s$-element set if $s<r$. For $s<r$, let $\mathcal{N}(K_s^{(r)},\cH)$ denote the collection of all $s$-subsets of $V(\cH)$.
Let $n,k,r,a$ be positive integers with $n \geq r \geq a\geq 1$. Define the $r$-graph
\begin{eqnarray*}
\mathcal{F}_{n,k,a}^{(r)}=\left\{F\in \binom{[n]}{r} : \big|F\cap[ak + a - 1]\big|\geq a\right\}.
\end{eqnarray*}

Liu and Wang \cite{LiWa} determined $\text{ex}_r(n,K_s^{(r)},M^r_{k+1})$ for all $s,r,k$ and sufficiently large $n$.

\begin{thm}[Liu and Wang \cite{LiWa}]\label{Liu-Wang-matching}
Let \( n, k, r, s \) be integers. 

\textbf{(i)} If \( r \leq s \leq k + r - 1 \) and \( n \geq 4(er)^{s-r+2}k \), then \( \ex_r(n, K_s^{(r)}, M^r_{k+1}) \leq \cN(K_s^{(r)},\mathcal{F}_{n,k,1}^{(r)}) \);

\textbf{(ii)} If \( k + r \leq s \leq (r-1)(k+1) \) and \( n \geq 4r^2k(er/(a-1))^{s-r+a} \), then \( \ex_r(n, K_s^{(r)}, M^r_{k+1}) \leq \cN(K_s^{(r)},\mathcal{F}_{n,k,a}^{(r)}) \), where \( a = \left\lfloor \frac{s-r}{k} \right\rfloor + 1 \);

\textbf{(iii)} If \( (r - 1)k + r \leq s \leq rk + r - 1 \) and \( n \geq rk + r - 1 \), then \( \ex_r(n, K_s^{(r)}, M^r_{k+1}) \leq \cN(K_s^{(r)},\mathcal{F}_{n,k,r}^{(r)}) \). 
\end{thm}

Note that $M^r_{k+1}=(k+1)K_r^{(r)}$. Gerbner \cite{Gerb1} considered a generalization of Theorem \ref{Liu-Wang-matching}, namely, the vertex-disjoint union of hypergraph cliques.
We introduce several definitions from \cite{Gerb1}. For two $r$-graphs $\mathcal{H}$ and $\mathcal{H}'$ with disjoint vertex sets, we denote by $\mathcal{H} + \mathcal{H}'$ the $r$-graph containing $\mathcal{H}$ and $\mathcal{H}'$, that contains as additional hyperedges all the $r$-sets intersecting both $\mathcal{H}$ and $\mathcal{H}'$. Let $x =\lceil \frac{kp - s}{k - 1}\rceil - 1$ and let $\mathcal{T}$ be a $K_{x+1}^{(r)}$-free $r$-graph on $n - k(p - x) + 1$ vertices with the most copies of $K_x^{(r)}$.

\begin{thm}[Gerbner \cite{Gerb1}]
Let $s \geq p \geq 2$ and $k \geq 1$ be arbitrary integers and let $x =\lceil \frac{kp - s}{k - 1}\rceil - 1$. Then we have $\ex_r(n, K_s^{(r)}, kK_p^{(r)})=(1+o(1))\mathcal{N}(K_s^{(r)}, K_{k(p-x)-1}^{(r)} + \mathcal{T})$. In particular, $\ex_r(n, K_s^{(r)}, kK_p^{(r)})=\Theta(n^x)$.
\end{thm}

Balko, Gerbner, Kang, Kim and Palmer \cite{BGKKP} established the asymptotics of the generalized Tur\'{a}n number for expansions of graphs $F$ with $\chi(F)>s>r$.

\begin{thm}[Balko, Gerbner, Kang, Kim and Palmer \cite{BGKKP}]\label{BGKKP}
Let $F$ be a graph and let $\chi(F)=k>s>r$. Then
$$\ex_r(n,K_s^{(r)},F^{r})=(1 + o(1))\cdot e(\cT_s(n,k-1)).$$
\end{thm}

Very recently, Axenovich, Gerbner, Liu and Patk\'{o}s \cite{AGLP} established a generalized hypergraph Tur\'{a}n result for expansions of paths and cycles.
For integers $r\geq 2$ and $t\geq 1$, let $S_{n,t}^r$ be the $n$-vertex $r$-graph consisting of all hyperedges that intersect a fixed set of $t$ vertices.

\begin{thm}[Axenovich, Gerbner, Liu and Patk\'{o}s \cite{AGLP}]\label{Gerbner-path}
Let $r \geq 3$ and $\ell \geq 4$ be integers and $\ell'=\lfloor\frac{\ell-1}{2}\rfloor$. Let $Q \in \{P_\ell^r, C_\ell^r\}$.

\textbf{(i)} If $s \leq \ell' + r - 1$, then $\ex_r(n, K_s^{(r)}, Q) = (1 + o(1))\mathcal{N}(K_s^{(r)}, S_{n,\ell'}^r)$.

\textbf{(ii)} If $s \geq \ell' + r$, then $\ex_r(n, K_s^{(r)}, Q) = o(n^{r-1})$.
\end{thm}

For general hypergraphs, Kirsch and Radcliffe \cite{KiRa} established bounds on the number of $s$-cliques in an $r$-graph $\mathcal{H}$ with bounded $i$-degree $\Delta$ (here, the \textit{$i$-degree} denotes the number of hyperedges containing a given $i$-vertex set), in three settings: $\mathcal{H}$ has $n$ vertices; $\mathcal{H}$ has $m$ hyperedges; and $\mathcal{H}$ has a fixed number of $t$-cliques for some $t$ satisfying $r \le t \le s$. In particular, for special values of $\Delta$, they characterized the extremal $r$-graphs and showed that the corresponding bounds are tight.

Let us now turn our attention to $r$-graphs that are not expansions. Let $S_2^{(3)}$ denote the unique $3$-graph on $4$ vertices with $2$ hyperedges. Bodn\'{a}r et al. \cite{BoChD} showed that determining the Tur\'{a}n number of a $3$-graph $\cF$ in the $\ell_2$-norm (see \cite{BaCL} for the relevant definition) is equivalent to maximizing $2\cN(S_2^{(3)},\cH)+3e(\cH)$ over all $\cF$-free $3$-graphs $\cH$. Their result implies that for sufficiently large $n$,
$$\ex_3(n,S_2^{(3)},K_4^{(3)}) = \cN(S_2^{(3)},\cC_n^{(3)}),$$
where $\cC_n^{(3)}$ denote the $3$-graph on $[n]$ consisting of all triples of types $V_1V_2V_3$, $V_1V_1V_2$, $V_2V_2V_3$ and $V_3V_3V_1$ for a balanced partition $V_1,V_2,V_3$ of $[n]$.

For $k\geq r$, let $K_{\ell_1,\ell_2,\dots,\ell_k}^{(r)}$ denote the complete $k$-partite $r$-graph with $\ell_1,\ell_2,\dots,\ell_k$ vertices in its parts. Ma, Yuan and Zhang \cite{MYZ} proved that for any $r$-graph $\cH$, there exists a constant $c>0$ such that for any integer $p\geq c$,
$$\ex_r(n,\cH,K_{s_1,s_2,\dots,s_{r-1},p}^{(r)}) = \Omega\left(n^{|V(\cH)|-\frac{e(\cH)}{s_1 s_2 \cdots s_{r-1}}}\right).$$
In particular, they established that for any positive integers $a_1,a_2,\dots,a_r,s_1,s_2,\dots,s_{r-1}$ and $r\geq2$ satisfying that $a_1<s_1\le s_2\le \dots\le s_{r-1}$ and $a_i\le s_{i-1}$ for each $2\le i\le r$, there exists a constant $c>0$ such that for any integer $p\geq c$,
$$\ex_r(n, K_{a_1,a_2,\dots,a_r}^{(r)}, K_{s_1,s_2,\dots,s_{r-1},p}^{(r)}) = \Theta\left(n^{\left(a_1+a_2+\dots+a_r\right)-\frac{a_1 a_2 \cdots a_r}{s_1 s_2 \cdots s_{r-1}}}\right).$$

Recently, Basu, R\"{o}dl and Zhao \cite{BRZ} showed that for any positive integers $a_1\leq a_2\leq\cdots\leq a_{r+1}$ and $r\geq2$,
$$\ex_r\big(n, K_{r+1}^{(r)}, K_{a_1,a_2,\dots,a_{r+1}}^{(r)}\big)=o\left(n^{r+1-\frac{1}{a_1a_2\cdots a_r}}\right).$$

To the best of our knowledge, besides the results mentioned above, the only other generalized Tur\'{a}n result on hypergraphs is that concerning complete bipartite $r$-graphs, which was introduced by Mubayi and Verstra\"{e}te \cite{MuV3}.

\begin{defi}[Complete bipartite $r$-graph]
Let $X_1,X_2,\dots,X_t$ be $t$ pairwise disjoint sets of size $r-1$, and let $Y$ be a set of $s$ elements, disjoint from $\bigcup_{i=1}^t X_i$. Then $K_{s,t}^{(r)}$ denotes the $r$-graph with vertex set $\bigcup_{i=1}^t X_i \cup Y$ and hyperedge set $\left\{X_i \cup \{y\}: i\in [t],\; y\in Y\right\}$.
\end{defi}

Note that $K_{s,t}^{(r)}$ is a hypergraph extension of the complete bipartite graph. In \cite{MuV3}, Mubayi and Verstra\"{e}te obtained some bounds on $\mathrm{ex}_r(n,K_{s,t}^{(r)})$ for $s \le t$. Subsequently, Xu, Zhang and Ge \cite{XZG} proved the following result.

\begin{thm}[Xu, Zhang and Ge \cite{XZG}]
Let $r\ge 3$. For any positive integer $t$ and any $r$-graph $\mathcal{H}$, there exists some constant $c$ which depends on $r$ and $t$ such that if $s\ge c$, then
$$\mathrm{ex}_r(n, \mathcal{H}, K_{s,t}^{(r)}) = \Omega\left(n^{|V(\cH)|-\frac{e(\cH)}{t}}\right).$$
\end{thm}
Furthermore, they showed that there exists some constant $c'_{r,t}$ which depends on $r$ and $t$ such that if $s\ge c'_{r,t}$ and $b<t$,
$$\ex_r(n, K_{1,b}^{(r)}, K_{s,t}^{(r)}) = \Theta\left(n^{1+b(r-1)-\frac{b}{t}}\right).$$

Motivated by the work of Alon and Shikhelman \cite{AlSh}, 
the goal of this paper is to provide a brief survey of generalized Tur\'{a}n problems for hypergraphs and to systematically investigate the generalized Tur\'{a}n problem for expansions.

\section{\normalsize Main results}
Given $r$-graphs $\cH$ and $\cF$, a generalized hypergraph Tur\'{a}n problem is called \textit{degenerate case} if the order of magnitude of the extremal function $\ex_r(n,\cH,\cF)$ is strictly smaller than the maximum possible, namely $\ex_r(n,\cH,\cF)=o(n^{|V(\cH)|})$. Otherwise, it is called \textit{non-degenerate case}, where $\ex_r(n,\cH,\cF)=\Theta(n^{|V(\cH)|})$. In the case $\cH=K_r^{(r)}$, $\ex_r(n,K_r^{(r)},F^r)={\rm ex}_r(n,F^r)=\Theta(n^r)$ if and only if $\chi(F)>r$, where $\chi(F)$ denotes the chromatic number of the graph $F$ \cite{Ge2}. In this work, we concentrate on the case where $\mathcal{H}=K_s^{(r)}$ (the complete $r$-graph on $s$ vertices). 

\subsection{\normalsize The non-degenerate case}
We first consider the non-degenerate case, for which we show that generalized Tur\'{a}n problems for expansions of graphs $F$ are non-degenerate if and only if $\chi(F)>s$. 

\begin{prop}\label{prop2.1}
Let $s\geq r$ be integers and $n$ be sufficiently large. Suppose $F$ is a graph. Then ${\rm ex}_r(n,K_s^{(r)},F^r)=\Theta(n^s)$ if and only if $\chi(F)>s$.
\end{prop}

Our first main result in the non-degenerate case is an exact determination of the generalized Tur\'{a}n number for expansions of complete graphs, 
generalizing the result of Pikhurko (see Theorem \ref{Pikhurko}).

\begin{thm}\label{thm0}
Let integers $\ell\geq s\geq r\geq3$. Then for sufficiently large $n$,
\begin{eqnarray*}
{\rm{ex}}_r(n,K_s^{(r)},K_{\ell+1}^r)= \cN(K_s^{(r)},\cT_r(n,\ell)).
\end{eqnarray*}
\end{thm}

Mubayi \cite{Mu} considered the Tur\'{a}n problem for the following family of $r$-graphs. For $t>r$, let $\cK_t^r$ denote the family of all $r$-graphs $\cF$ with at most $\binom{t}{2}$ hyperedges for which there exists a set $S$ of $t$ vertices, such that every pair $u,v\in S$ is contained in some hyperedge of $\cF$. When $r=2$, the family $\cK_t^r$ reduces to $K_t$, however when $r>2$, it contains more than one $r$-graph.
Let $\ex_r(n,K_s^{(r)},\cK_{t}^r)$ denote the maximum number of copies of $K_s^{(r)}$ in an $n$-vertex $r$-graph that contains no subhypergraph isomorphic to any member of $\cK_{t}^r$. Mubayi \cite{Mu} showed that $\ex_r(n,K_r^{(r)},\cK_{\ell+1}^r)=t_r(n,\ell)$, with the unique extremal $r$-graph being $\cT_r(n,\ell)$. Since $K_{\ell+1}^r\in \cK_{\ell+1}^r$ and $\cT_r(n,\ell)$ is $\cK_{\ell+1}^r$-free, applying Theorem \ref{thm0} yields the following result.

\begin{cor}\label{cor0}
Let integers $\ell\geq s\geq r\geq3$. Then for sufficiently large $n$,
\begin{eqnarray*}
\ex_r(n,K_s^{(r)},\cK_{\ell+1}^r)=\cN(K_s^{(r)},\cT_r(n,\ell)).
\end{eqnarray*}
\end{cor}

Since $K^r_{\ell+1}$ and $\cK^r_{\ell+1}$ are natural generalizations of the complete graph $K_{\ell+1}$, Theorem~\ref{thm0} and Corollary~\ref{cor0} may be regarded as hypergraph extensions of the results of Zykov~\cite{Zyko} and Erd\H{o}s~\cite{Erdo}.

Our second main result establishs the asymptotics of the generalized Tur\'{a}n number for expansions of the vertex-disjoint union of complete graphs. Recall that $S_{n,t}^r(n-t,\ell)$ is obtained from $\cT_r(n-t,\ell)$ by adding a set $U$ of $t$ new vertices and adding all $r$-sets that contain at least one vertex from $U$ as hyperedges.

\begin{thm}\label{thm00}
Let integers $k\geq1$ and $\ell\geq s\geq r\geq3$. Then for sufficiently large $n$,
\begin{eqnarray*}
{\rm{ex}}_r(n,K_s^{(r)},kK_{\ell+1}^r)= (1+o(1))\cN\big(K_s^{(r)},S_{n,k-1}^r(n-k+1,\ell)\big).
\end{eqnarray*}
\end{thm}

We remark that the condition $\ell\geq s$ in Theorems \ref{thm0} and \ref{thm00} is necessary. Indeed, if $\ell<s$, then these two problems become degenerate, as follows directly from Proposition \ref{prop2.1}.


\subsection{\normalsize The degenerate case}
Let us turn to the degenerate case. We focus on generalized Tur\'{a}n problems for expansions of several classes of forests, and determine the asymptotics of generalized Tur\'{a}n numbers for these classes of hypergraph forests. For non-negative integers $a$ and $b$, let $\binom{a}{b}$ denote the binomial coefficient. In particular, we define $\binom{a}{b}=0$ if $a<b$, and $\binom{0}{0}=1$.

Our first two main results in this subsection concern expansions of star forests. First we establish an upper bound for ${\rm{ex}}_r(n,K_s^{(r)},S_\ell^r)$, which generalizes the Duke-Erd\H{o}s theorem (see Theorem \ref{lem3.2}).

\begin{thm}\label{thm1}
Fix integers $s\geq r\geq3$ and $\ell\geq2$. Then there exists a constant $c(s,r,\ell)$ such that for sufficiently large $n$,
\begin{eqnarray*}
{\rm{ex}}_r(n,K_s^{(r)},S_\ell^r)\leq c(s,r,\ell)n^{r-2}.
\end{eqnarray*}
Moreover, if $s\leq \ell+r-2$, then $\ex_r(n,K_s^{(r)},S_\ell^r)=\Theta(n^{r-2})$.
\end{thm}

In the general case, we determine the asymptotics of the generalized Tur\'{a}n number for expansions of star forests. Given an $r$-graph $\cH$, let $\cF_r(n,\cH)$ denote the family of all $\cH$-free $r$-graphs on $n$ vertices.

\begin{thm}\label{thm2}
Let $r\geq3$, $k\geq1$ and $\ell_1\geq\cdots\geq \ell_k\geq2$ be integers and $n$ be sufficiently large.

\textbf{(i)} If $s \leq k+r-2$, then
\begin{eqnarray*}
{\rm{ex}}_r(n,K_s^{(r)},S^r_{\ell_1}\cup\cdots\cup S^r_{\ell_k})= 
\sum_{i=0}^{r-1}\binom{k-1}{s-i}\binom{n-k+1}{i}+\max_{\cG\in \cF_r(n-k+1,S^r_{\ell_k})}\sum_{i=r}^s\binom{k-1}{s-i}\cN(K_i^{(r)},\cG).
\end{eqnarray*}

\textbf{(ii)} If $s \geq k+r-1$, then
\begin{eqnarray*}
{\rm{ex}}_r(n,K_s^{(r)},S^r_{\ell_1}\cup\cdots\cup S^r_{\ell_k})= O(n^{r-2}).
\end{eqnarray*}
Moreover, if $k+r-1\leq s\leq \sum_{i=1}^k(\ell_i+1)+r-3$, then ${\rm{ex}}_r(n,K_s^{(r)},S^r_{\ell_1}\cup\cdots\cup S^r_{\ell_k})=\Theta(n^{r-2})$.
\end{thm}

Theorem \ref{thm2} strengthens the result of Khormali and Palmer \cite{KP} concerning star forests (see Theorem \ref{Khormali-Palmer}). Indeed, when $s=r$, we have $\sum_{i=0}^{r-1}\binom{k-1}{s-i}\binom{n-k+1}{i}=\binom{n}{r}-\binom{n-k+1}{r}$ and $\max_{\cG\in \cF_r(n-k+1,S^r_{\ell_k})}\sum_{i=r}^s\binom{k-1}{s-i}\cN(K_i^{(r)},\cG)={\rm{ex}}_r(n-k+1,S^r_{\ell_k})$. Therefore, we have the following Tur\'{a}n result for expansions of star forests.

\begin{cor}\label{cor1}
Let $r\geq3$, $k\geq1$ and $\ell_1\geq\cdots\geq \ell_k\geq2$ be integers and $n$ be sufficiently large. Then
\begin{eqnarray*}
{\rm{ex}}_r(n,K_r^{(r)},S^r_{\ell_1}\cup\cdots\cup S^r_{\ell_k})
=\binom{n}{r}-\binom{n-k+1}{r}+{\rm{ex}}_r(n-k+1,S^r_{\ell_k}).
\end{eqnarray*}
\end{cor}

Our third main result establishes the asymptotics of generalized Tur\'{a}n numbers for expansions of linear forests and the vertex-disjoint union of cycles. We denoted by $C_\ell$ a cycle with $\ell$ edges.

\begin{thm}\label{thm1.1}
Let $r\geq3$, $k\geq1$ and $\ell_1,\dots,\ell_k\geq 5$ be integers and $t=\sum_{i=1}^k\lfloor\frac{\ell_i+1}{2}\rfloor-1$. Suppse $Q_i\in\{P_{\ell_i}^r,C_{\ell_i}^r\}$, $Q_1\cup\cdots\cup Q_k\in \{P_{\ell_1}^r\cup\cdots\cup P_{\ell_k}^r, C_{\ell_1}^r\cup\cdots\cup C_{\ell_k}^r\}$ and $n$ is sufficiently large.

\textbf{(i)} If $s \leq t+r-1$, then
\begin{eqnarray*}
{\rm{ex}}_r(n,K_s^{(r)},Q_1\cup\cdots\cup Q_k)= (1+o(1))\cN(K_s^{(r)},S_{n,t}^r)=\binom{t}{s-r+1}\frac{n^{r-1}}{(r-1)!}+o(n^{r-1}).
\end{eqnarray*}

\textbf{(ii)} If $s \geq t+r$, then
\begin{eqnarray*}
{\rm{ex}}_r(n,K_s^{(r)},Q_1\cup\cdots\cup Q_k)= o(n^{r-1}).
\end{eqnarray*}
\end{thm}

In the case where $\ell_1=\cdots=\ell_k$, Theorem \ref{thm1.1} immediately yields the following corollary.

\begin{cor}\label{cor2}
Let $r\geq3$, $k\geq1$ and $\ell\geq 5$ be integers and $\ell'=\lfloor\frac{\ell+1}{2}\rfloor$. Suppose $Q\in\{P_{\ell}^r,C_{\ell}^r\}$ and $n$ is sufficiently large.

\textbf{(i)} If $s\leq k\ell'+r-2$, then ${\rm{ex}}_r(n,K_s^{(r)},kQ)=(1+o(1))\cN(K_s^{(r)},S_{n,k\ell'-1}^r)=\binom{k\ell'-1}{s-r+1}\frac{n^{r-1}}{(r-1)!}+o(n^{r-1})$.

\textbf{(ii)} If $s \geq k\ell'+r-1$, then ${\rm{ex}}_r(n,K_s^{(r)},kQ)=o(n^{r-1})$.
\end{cor}

Moreover, we determine the asymptotics of generalized Tur\'{a}n numbers for two classes of expansions: those of star-path forests and those of the vertex-disjoint union of cycles and stars, as follows.

\begin{thm}\label{thm3}
Let integers $r\geq3$, $k\geq1$, $p\geq0$, $\ell_1,\dots,\ell_k\geq 5$, $t_1,\dots,t_{p}\geq2$ if $p\geq1$, and $t=\sum_{i=1}^k\lfloor\frac{\ell_i+1}{2}\rfloor+p-1$. Suppose $Q_i\in\{P_{\ell_i}^r,C_{\ell_i}^r\}$, $Q_1\cup\cdots\cup Q_k\in \{P_{\ell_1}^r\cup\cdots\cup P_{\ell_k}^r, C_{\ell_1}^r\cup\cdots\cup C_{\ell_k}^r\}$ and $n$ is sufficiently large.

\textbf{(i)} If $s \leq t+r-1$, then
\begin{eqnarray*}
{\rm{ex}}_r(n,K_s^{(r)},Q_1\cup\cdots\cup Q_k\cup S^r_{t_1}\cup\cdots\cup S^r_{t_p})= (1+o(1))\cN(K_s^{(r)},S_{n,t}^r)=\binom{t}{s-r+1}\frac{n^{r-1}}{(r-1)!}+o(n^{r-1}).
\end{eqnarray*}

\textbf{(ii)} If $s \geq t+r$, then
\begin{eqnarray*}
{\rm{ex}}_r(n,K_s^{(r)},Q_1\cup\cdots\cup Q_k\cup S^r_{t_1}\cup\cdots\cup S^r_{t_p})= o(n^{r-1}).
\end{eqnarray*}
\end{thm}

Note that in Theorems \ref{thm1.1} and \ref{thm3}, the condition $\ell_1,\dots,\ell_k\geq 5$ can be relaxed to $\ell_1,\dots,\ell_k\geq 4$ when $Q_1\cup\cdots\cup Q_k=P_{\ell_1}^r\cup\cdots\cup P_{\ell_k}^r$.

The rest of this paper is organized as follows. In Section \ref{3}, we present some notations and definitions, and provide the proofs of Proposition \ref{prop2.1}, Theorems \ref{thm0} and \ref{thm00}. In particular, we establish a general upper bound on generalized Tur\'{a}n numbers for expansions (see Lemma \ref{lem0}) in Section \ref{3}.
The proofs of Theorems \ref{thm1} and \ref{thm2} are presented in Section \ref{4}, and the proofs of Theorems \ref{thm1.1} and \ref{thm3} are presented in Section \ref{5}. Finally, we state some concluding remarks in Section \ref{6}, including several more general results.

\section{\normalsize Proofs of Proposition \ref{prop2.1}, Theorems \ref{thm0} and \ref{thm00}}\label{3}
Let $\cH=(V(\cH),E(\cH))$ be an $r$-graph. A \textit{strong $k$-coloring} of $\cH$ is a partition $(C_1, C_2, \ldots, C_k)$ of $V(\cH)$ such that the same color does not appear twice in the same hyperedge \cite{Bre}. In other words,
\begin{eqnarray*}
|e \cap C_i| \leq 1
\end{eqnarray*}
for any $e\in E(\cH)$ and $i\in [k]$. 
The \textit{strong chromatic number} of $\cH$, denoted by $\chi(\cH)$, is the smallest $k$ such that $\cH$ has a strong $k$-coloring. Clearly, $\chi(F^r)\geq\chi(F)$ for any graph $F$. 

Let $U\subseteq V(\cH)$ be a nonempty subset. Let $\cH-U$ denote the $r$-graph obtained from $\cH$ by removing the vertices in $U$ and the hyperedges incident to them. Let $\cH[U]$ denote the subhypergraph of $\cH$ induced by $U$.

The \textit{link hypergraph} of a vertex $v$ of $\cH$ is an $(r-1)$-graph, defined as
\begin{eqnarray*}
\cL_v=\big\{e\backslash\{v\}\,\big|\,e\in E(\cH), v\in e\big\}.
\end{eqnarray*}
The \textit{degree} $d_\cH(v)$ of $v$ is defined as the number of hyperedges in $\cL_v$, i.e., the number of hyperedges in $\cH$ that contain $v$.

Recall that $K_s^{(r)}$ denotes the complete $r$-graph on $s$ vertices if $s\geq r$, and an $s$-element set if $s<r$. Given two $r$-graphs $\cG_1$ and $\cG_2$, $\mathcal{N}(\cG_1,\cG_2)$ denotes the number of copies of $\cG_1$ in $\cG_2$. In particular, $\mathcal{N}(K_s^{(r)},\cH)$ denotes the number of copies of $K_s^{(r)}$ in $\cH$ if $s\geq r$, and the collection of all $s$-subsets of $V(\cH)$ if $s<r$.
For any $v\in V(\cH)$ and $s\geq1$, let
$$\cK_{v}(K_s^{(r)},\cH)=\{K_s^{(r)}\subseteq \cH\,|\,v\in V(K_s^{(r)})\},$$
namely, the collection of all $K_s^{(r)}$ in $\cH$ containing the vertex $v$.

For $s\geq r$, the \textit{$s$-clique hypergraph} of $\cH$ is an $s$-graph with vertex set $V(\cH)$, where an $s$-subset of $V(\cH)$ is a hyperedge if and only if it induces a copy of $K_s^{(r)}$ in $\cH$.

Now let us start with the proof of Proposition \ref{prop2.1} that we restate here for convenience.

\begin{propbis}{prop2.1}
Let $s\geq r$ be integers and $n$ be sufficiently large. Suppose $F$ is a graph. Then ${\rm ex}_r(n,K_s^{(r)},F^r)=\Theta(n^s)$ if and only if $\chi(F)>s$.
\end{propbis}
\begin{proof}[{\bf{Proof.}}]
Recall that $\chi(F^r)\geq\chi(F)$. Observe that we may color the $r-2$ new vertices in each hyperedge of $F^r$ using $r-2$ distinct colors (chosen from at most $\max\{\chi(F),r\}$ available colors), all of which differ from the colors assigned to two vertices of the edge of $F$. Then $\chi(F^r)\leq \max\{\chi(F),r\}$. Therefore,
\begin{eqnarray}
\chi(F)\leq \chi(F^r) \leq \max\{\chi(F),r\}.
\end{eqnarray}

($\Leftarrow$). Suppose $F$ is a graph with $\chi(F)>s\geq r$. Then $\chi(F^r)=\chi(F)>s$ by (3.1). This implies that $\cT_r(n,s)$ is $F^r$-free as $\chi(\cT_r(n,s))=s$, and thus ${\rm ex}_r(n,K_s^{(r)},F^r)\geq \cN(K_s^{(r)},\cT_r(n,s))=\Theta(n^s)$. Note that ${\rm ex}_r(n,K_s^{(r)},F^r)=O(n^s)$. Thus, ${\rm ex}_r(n,K_s^{(r)},F^r)=\Theta(n^s)$.

($\Rightarrow$). Suppose now that $F$ is a graph such that ${\rm ex}_r(n,K_s^{(r)},F^r)=\Theta(n^s)$. Let $\cH$ be an $F^r$-free $r$-graph on $n$ vertices with $\cN(K_s^{(r)},\cH)={\rm ex}_r(n,K_s^{(r)},F^r)=\Theta(n^s)$. Observe that $\chi(\cH)=\ell\geq s$. Indeed, if $\ell<s$, then any $s$-subset of $V(\cH)$ contains two vertices of the same color. This implies that no $s$-subset of $V(\cH)$ induces a copy of $K_s^{(r)}$. Thus, $\cN(K_s^{(r)},\cH)=0$, which is a contradiction. We claim that among the $\ell$ parts of $\cH$, at least $s$ of them have size $\Theta(n)$. Otherwise, $\cN(K_s^{(r)},\cH)=o(n^s)$, which is a contradiction.
Then $\cH$ contains any $s$-partite $r$-graph, as $n$ is sufficiently large. This implies that $\chi(F^r)>s$, since $\cH$ is $F^r$-free. If $\chi(F)\leq s$, then $\chi(F^r)\leq \max\{\chi(F),r\}\leq s$ by (3.1), which contradicts $\chi(F^r)>s$. Thus, $\chi(F)>s$, completing the proof.
\end{proof}

Before proceeding to the proofs of Theorems \ref{thm0} and \ref{thm00}, we present a lemma that establishes a general upper bound on ${\rm{ex}}_r(n,K_s^{(r)},F^r)$ for any graph $F$.

\begin{lem}\label{lem0}
Let $s\geq r$ be integers and $F$ be an arbitrary graph. Suppose $\cH$ is an $F^r$-free $r$-graph on $n$ vertices. Then $\cN(K_s^{(r)},\cH)\leq \ex_s(n,F^s)$. In particular, ${\rm{ex}}_r(n,K_s^{(r)},F^r)\leq \ex_s(n,F^s)$.
\end{lem}

Note that $\cN(K_s^{(r)},\cT_r(n,\ell))=t_s(n,\ell)$. Hence, Theorems \ref{Pikhurko} and \ref{thm0} imply that the upper bound in Lemma \ref{lem0} is sharp for the case when $F=K_{\ell+1}$ and $\ell\geq s$. 

\begin{proof}[\bf Proof]
Denote by $\cH^s$ the $s$-clique hypergraph of $\cH$. We claim that $\cH^s$ is $F^s$-free. Otherwise, assume that $\cH^s$ contains a copy of $F^s$, denoted by $K$.
By the definition of the $s$-clique hypergraph, every $r$-subset of each hyperedge of $\cH^s$ is a hyperedge of $\cH$. Then every $r$-subset of each hyperedge of $K$ is also a hyperedge of $\cH$. Now we delete $s-r$ vertices of degree one from each hyperedge of $K$, which yields a copy of $F^r$, denoted by $K'$. Since $K'$ consists of $r$-subsets of hyperedges of $K$, $K'$ is contained in $\cH$, which contradicts the fact that $\cH$ is $F^r$-free. Thus, $\cN(K_s^{(r)},\cH)=e(\cH^s)\leq \ex_s(n,F^s)$.
\end{proof}

We now present the proof of Theorem \ref{thm0}. Recall that it states that ${\rm{ex}}_r(n,K_s^{(r)},K_{\ell+1}^r)= \cN(K_s^{(r)},\cT_r(n,\ell))$ for $\ell\geq s\geq r\geq3$ and sufficiently large $n$. The lower bound is trivially achieved by $\cT_r(n,\ell)$. For the upper bound, combining Lemma \ref{lem0} and Theorem \ref{Pikhurko}, we obtain ${\rm{ex}}_r(n,K_s^{(r)},K_{\ell+1}^r)\leq \ex_s(n,K_{\ell+1}^s)=t_s(n,\ell)=\cN(K_s^{(r)},\cT_r(n,\ell))$.

Recall that $S_{n,t}^r(n-t,\ell)$ is obtained from $\cT_r(n-t,\ell)$ by adding a set $U$ of $t$ new vertices and adding all $r$-sets that contain at least one vertex from $U$ as hyperedges. We now begin with the proof of Theorem \ref{thm00} that we restate here for convenience.

\begin{thmbis}{thm00}
Let integers $k\geq1$ and $\ell\geq s\geq r\geq3$. Then for sufficiently large $n$,
\begin{eqnarray*}
{\rm{ex}}_r(n,K_s^{(r)},kK_{\ell+1}^r)= (1+o(1))\cN\big(K_s^{(r)},S_{n,k-1}^r(n-k+1,\ell)\big).
\end{eqnarray*}
\end{thmbis}
\begin{proof}[{\bf{Proof.}}]
Note that $\cT_r(n-k+1,\ell)$ is $K_{\ell+1}^r$-free. Hence, $S_{n,k-1}^r(n-k+1,\ell)$ is $kK_{\ell+1}^r$-free. Then ${\rm{ex}}_r(n,K_s^{(r)},kK_{\ell+1}^r)\geq \cN\big(K_s^{(r)},S_{n,k-1}^r(n-k+1,\ell)\big)$.
We now consider the upper bound. Let $\cH$ be a $kK_{\ell+1}^r$-free $r$-graph on $n$ vertices. 
Since $\ell\geq s\geq r\geq3$, by Lemma \ref{lem0} and Theorem \ref{union-complete-graphs}, we have
$$\cN(K_s^{(r)},\cH)\leq \ex_s(n,kK_{\ell+1}^s)=e\big(S_{n,k-1}^s(n-k+1,\ell)\big)=\sum_{i=0}^{s-1}\binom{k-1}{s-i}\binom{n-k+1}{i}+t_s(n-k+1,\ell).$$

Note that $\cN\big(K_s^{(r)},\cT_r(n-k+1,\ell)\big)=t_s(n-k+1,\ell)$. Then 
\begin{eqnarray*}
\cN\big(K_s^{(r)},S_{n,k-1}^r(n-k+1,\ell)\big)&=& \sum_{i=0}^{s}\binom{k-1}{s-i}\cN\big(K_i^{(r)},\cT_r(n-k+1,\ell)\big) \\
&=& \sum_{i=0}^{r-1}\binom{k-1}{s-i}\binom{n-k+1}{i}+\sum_{i=r}^{s-1}\binom{k-1}{s-i}\cN\big(K_i^{(r)},\cT_r(n-k+1,\ell)\big) \\
&{}& +t_s(n-k+1,\ell).
\end{eqnarray*}
Thus,
$$\cN(K_s^{(r)},\cH)\leq \sum_{i=0}^{s-1}\binom{k-1}{s-i}\binom{n-k+1}{i}+t_s(n-k+1,\ell)=(1+o(1))\cN\big(K_s^{(r)},S_{n,k-1}^r(n-k+1,\ell)\big).$$
This completes the proof.
\end{proof}

\section{\normalsize Proofs of Theorems \ref{thm1} and \ref{thm2}}\label{4}
Now let us start with the proof of Theorem \ref{thm1} that we restate here for convenience.

\begin{thmbis}{thm1}
Fix integers $s\geq r\geq3$ and $\ell\geq2$. Then there exists a constant $c(s,r,\ell)$ such that for sufficiently large $n$,
\begin{eqnarray*}
{\rm{ex}}_r(n,K_s^{(r)},S_\ell^r)\leq c(s,r,\ell)n^{r-2}.
\end{eqnarray*}
Moreover, if $s\leq \ell+r-2$, then $\ex_r(n,K_s^{(r)},S_\ell^r)=\Theta(n^{r-2})$.
\end{thmbis}

\begin{proof}[{\bf{Proof.}}]
Let $\cH$ be an $r$-graph on $n$ vertices with $\cN(K_s^{(r)},\cH)> c(s,r,\ell)n^{r-2}$ for some sufficiently large constant $c(s,r,\ell)$. We will show that $\cH$ contains a copy of $S_\ell^r$. 

\begin{clm}\label{clm1}
$|\cK_{v}(K_s^{(r)},\cH)|=O(n^{r-2})$ for any $v\in V(\cH)$.
\end{clm}
\begin{proof}[\bf Proof of Claim]
Suppose to the contrary that there exists a vertex $u\in V(\cH)$ such that $|\cK_{u}(K_s^{(r)},\cH)|=\omega(n^{r-2})$. This implies that there are $\omega(n^{r-2})$ copies of $K_{s-1}^{(r)}$ such that each $K_{s-1}^{(r)}$ forms a $K_{s}^{(r)}$ together with $u$. Denote by $X_{s-1}$ the collection of such $K_{s-1}^{(r)}$. Then $|X_{s-1}|=\omega(n^{r-2})$.

Let $\cS^{r-1}=\big\{K_{r-1}^{(r)}\ \big|\ K_{r-1}^{(r)}\subseteq V(K_{s-1}^{(r)}) \ \text{and}\ K_{s-1}^{(r)}\in X_{s-1}\big\}$. Consider the $(r-1)$-graph $\cH^{r-1}=\Big(\bigcup_{K_{s-1}^{(r)}\in X_{s-1}}V(K_{s-1}^{(r)}),\cS^{r-1}\Big)$. 
We claim that $\cH^{r-1}$ is $M_{\ell}^{r-1}$-free. Otherwise, assume that $\cH^{r-1}$ contains a copy of $M_{\ell}^{r-1}$. By the definitions of $X_{s-1}$ and $\cS^{r-1}$, the vertex $u$ forms a hyperedge of $\cH$ together with any $(r-1)$-edge of $\cH^{r-1}$. Thus, $\cH[\{u\}\cup V(M_{\ell}^{r-1})]$ contains a copy of $S_{\ell}^{r}$ with $u$ as its center (i.e., the unique vertex of degree not equal to one), and we are done.

Since $\cH^{r-1}$ is $M_{\ell}^{r-1}$-free, by Theorem \ref{Liu-Wang-matching}, we have $\cN(K_{s-1}^{(r-1)},\cH^{r-1})=O(n^{r-2})$. By the definitions of $\cS^{r-1}$ and $\cH^{r-1}$, it follows that $|X_{s-1}|\leq \cN(K_{s-1}^{(r-1)},\cH^{r-1})$, since each $K_{s-1}^{(r)}\in X_{s-1}$ corresponds to a copy of $K_{s-1}^{(r-1)}$ in $\cH^{r-1}$. Thus, $|X_{s-1}|=O(n^{r-2})$, which contradicts $|X_{s-1}|=\omega(n^{r-2})$. This completes the proof of Claim \ref{clm1}.
\end{proof}

Let $u\in V(\cH)$ be a vertex such that $|\cK_{u}(K_s^{(r)},\cH)|=\max_{v\in V(\cH)}|\cK_{v}(K_s^{(r)},\cH)|$. By Claim \ref{clm1}, we may assume that $|\cK_{u}(K_s^{(r)},\cH)|=\lambda n^{r-3}$, where $\lambda=O(n)$. By the definitions of $\cK_{u}(K_s^{(r)},\cH)$ and $\cL_u$, it follows that $\bigcup_{K_{s}^{(r)}\in \cK_{u}(K_s^{(r)},\cH)}V(K_{s}^{(r)})\backslash\{u\}\subseteq V(\cL_u)$, and $K_{r-1}^{(r)}\in E(\cL_u)$ for any $K_{r-1}^{(r)}\in V(K_{s}^{(r)})\backslash\{u\}$, where $K_{s}^{(r)}\in \cK_{u}(K_s^{(r)},\cH)$.
Let $M_t^{r-1}$ be a maximum matching in $\cL_u$. Clearly, $t\leq \ell-1$. Otherwise, we may find a copy of $S_{\ell}^{r}$ with $u$ as its center, and we are done. Therefore, each $(r-1)$-edge of $\cL_u$ intersects $V(M_t^{r-1})$ in at least one vertex. This implies that $\big|\big(V(K_{s}^{(r)})\backslash\{u\}\big)\cap V(M_t^{r-1})\big|\geq s-r+1$ for any $K_{s}^{(r)}\in \cK_{u}(K_s^{(r)},\cH)$. Otherwise, we may find a copy of $M_{t+1}^{r-1}$ in $\cL_u$, which is a contradiction.

Since $|V(M_t^{r-1})|\leq (r-1)(\ell-1)$, there are at most $\binom{(r-1)(\ell-1)}{s-r+1}$ subsets of size $s-r+1$ in $V(M_t^{r-1})$. By the pigeonhole principle, there exists an $(s-r+1)$-set $S_{s-r+1}\subseteq V(M_t^{r-1})$ such that there are at least $\lambda n^{r-3}/\binom{(r-1)(\ell-1)}{s-r+1}$ copies of $K_{s}^{(r)}$ in $\cK_{u}(K_s^{(r)},\cH)$ that contain $S_{s-r+1}$ (precisely, $S_{s-r+1}\cup \{u\}$). Note that each such $K_{s}^{(r)}$ corresponds to a distinct $(r-2)$-set $V(K_{s}^{(r)})\backslash(S_{s-r+1}\cup\{u\})$ in $V(\cL_u)\backslash(S_{s-r+1}\cup\{u\})$.
Clearly, there are at most $\binom{n-s+r-2}{r-3}\leq n^{r-3}$ subsets of size $r-3$ in $V(\cH)\backslash(S_{s-r+1}\cup\{u\})$. By the pigeonhole principle again, there exists an $(r-3)$-set $S_{r-3}\subseteq V(\cL_u)\backslash S_{s-r+1}$ such that there are at least $\lambda/\binom{(r-1)(\ell-1)}{s-r+1}$ copies of $K_{s}^{(r)}$ in $\cK_{u}(K_s^{(r)},\cH)$ that contain $S_{s-r+1}\cup S_{r-3}$ (precisely, $S_{s-r+1}\cup S_{r-3}\cup \{u\}$).

Now we delete the vertex set $S_{s-r+1}\cup S_{r-3}\cup \{u\}$ from $\cH$. Consequently, at most $(s-1)\lambda n^{r-3}$ copies of $K_{s}^{(r)}$ in $\cH$ are removed, and at least $\lambda/\binom{(r-1)(\ell-1)}{s-r+1}$ copies of $K_{s}^{(r)}$ in $\cK_{u}(K_s^{(r)},\cH)$ containing $S_{s-r+1}\cup S_{r-3}\cup \{u\}$ are deleted. Since $c(s,r,\ell)$ is sufficiently large, we may repeat this procedure (more than $\ell$ times) until $c(s,r,\ell)n^{r-2}/2$ copies of $K_{s}^{(r)}$ in $\cH$ have been deleted. More precisely, we repeat this procedure at least $c(s,r,\ell)n^{r-2}/(2(s-1)\lambda n^{r-3})=c(s,r,\ell)n/(2(s-1)\lambda)$ times. In the process, we delete at least $c(s,r,\ell)n/(2(s-1)\lambda)$ distinct $(s-1)$-sets, and thus remove at least $\frac{c(s,r,\ell)n}{2(s-1)\lambda}\cdot \lambda\cdot \binom{(r-1)(\ell-1)}{s-r+1}^{-1}=\frac{c(s,r,\ell)n}{2(s-1)}\cdot \binom{(r-1)(\ell-1)}{s-r+1}^{-1}$ copies of $K_{s}^{(r)}$ in $\cH$ containing some deleted $(s-1)$-set.
Without loss of generality, let $Y_{s-1}$ be the collection of all deleted $(s-1)$-sets. Then $|Y_{s-1}|\geq \frac{c(s,r,\ell)n}{2(s-1)\lambda}>\ell$.

We claim that there exists a vertex $w\in V(\cH)$ such that $w$ forms a copy of $K_{s}^{(r)}$ in $\cH$ with each of $\ell$ distinct $(s-1)$-sets in $Y_{s-1}$. Otherwise, every vertex in $V(\cH)$ forms a copy of $K_{s}^{(r)}$ in $\cH$ with at most $\ell-1$ distinct $(s-1)$-sets in $Y_{s-1}$. This implies that we can only remove at most $(\ell-1)n$ copies of $K_{s}^{(r)}$ in $\cH$ containing some $(s-1)$-set in $Y_{s-1}$. However, it follows that $\frac{c(s,r,\ell)}{2(s-1)}\cdot \binom{(r-1)(\ell-1)}{s-r+1}^{-1}>\ell-1$ as $c(s,r,\ell)$ is sufficiently large, which is a contradiction. Thus, by selecting one $(r-1)$-subset from each of $\ell$ distinct $(s-1)$-sets in $Y_{s-1}$, we may obtain a copy of $S_{\ell}^{r}$ with $w$ as its center, completing the proof.

If $s\leq \ell+r-2$, then we consider the following $r$-graph $\cH_1$: take a vertex set partition $A\cup B$ where $|A|=\ell$ and $|B|=n-\ell$, then add as hyperedges all $r$-sets containing at least two vertices from $A$. We claim that $\cH_1$ is $S_\ell^r$-free. Suppose to the contrary that $\cH_1$ contains a copy of $S_\ell^r$. Let $u\in V(S_\ell^r)$ denote the center of this $S_\ell^r$. If $u\in A$, then by the definition of $\cH_1$, $|V(S_\ell^r)\cap A|\geq \ell+1$, contradicting $|A|=\ell$. If $u\in B$, then by the definition of $\cH_1$, $|V(S_\ell^r)\cap A|\geq 2\ell$, which also contradicts $|A|=\ell$. Therefore, $\cH_1$ is $S_\ell^r$-free. Then
$$\ex_r(n,K_s^{(r)},S_\ell^r)\geq \cN(K_s^{(r)},\cH_1)=\sum_{i=0}^{r-2}\binom{\ell}{s-r+2+i}\binom{n-\ell}{r-2-i}.$$
Combining the upper bound, we have $\ex_r(n,K_s^{(r)},S_\ell^r)=\Theta(n^{r-2})$. 
\end{proof}

Recall that for an $r$-graph $\cH$, $\cF_r(n,\cH)$ denotes the family of all $\cH$-free $r$-graphs on $n$ vertices. We now continue with the proof of Theorem \ref{thm2} that we restate here for convenience.

\begin{thmbis}{thm2}
Let $r\geq3$, $k\geq1$ and $\ell_1\geq\cdots\geq \ell_k\geq2$ be integers and $n$ be sufficiently large.

\textbf{(i)} If $s \leq k+r-2$, then
\begin{eqnarray*}
{\rm{ex}}_r(n,K_s^{(r)},S^r_{\ell_1}\cup\cdots\cup S^r_{\ell_k})= 
\sum_{i=0}^{r-1}\binom{k-1}{s-i}\binom{n-k+1}{i}+\max_{\cG\in \cF_r(n-k+1,S^r_{\ell_k})}\sum_{i=r}^s\binom{k-1}{s-i}\cN(K_i^{(r)},\cG).
\end{eqnarray*}

\textbf{(ii)} If $s \geq k+r-1$, then
\begin{eqnarray*}
{\rm{ex}}_r(n,K_s^{(r)},S^r_{\ell_1}\cup\cdots\cup S^r_{\ell_k})= O(n^{r-2}).
\end{eqnarray*}
Moreover, if $k+r-1\leq s\leq \sum_{i=1}^k(\ell_i+1)+r-3$, then ${\rm{ex}}_r(n,K_s^{(r)},S^r_{\ell_1}\cup\cdots\cup S^r_{\ell_k})=\Theta(n^{r-2})$.
\end{thmbis}

\begin{proof}[{\bf{Proof.}}]
For the lower bound in \textbf{(i)}, we consider the following $r$-graph. Let us take any $S^r_{\ell_k}$-free $r$-graph $\cG$ on $n-k+1$ vertices and a set $U$ of $k-1$ vertices, then add each hyperedge that contains at least one vertex of $U$. Since $\cG$ is $S^r_{\ell_k}$-free and $\ell_1\geq\cdots\geq \ell_k$, any $S_{\ell_i}^r$ in the resulting $r$-graph contains at least one vertex in $U$. Thus, the resulting $r$-graph is $S^r_{\ell_1}\cup\cdots\cup S^r_{\ell_k}$-free as $|U|<k$. 
Clearly, the number of copies of $K_s^{(r)}$ in the resulting $r$-graph is $\sum_{i=0}^{r-1}\binom{k-1}{s-i}\binom{n-k+1}{i}+\sum_{i=r}^s\binom{k-1}{s-i}\cN(K_i^{(r)},\cG)$.

We now continue with the upper bound in \textbf{(i)}. 
Suppose $\cH$ is an $r$-graph on $n$ vertices with
\begin{eqnarray*}
\cN(K_s^{(r)},\cH)> \sum_{i=0}^{r-1}\binom{k-1}{s-i}\binom{n-k+1}{i}+\max_{\cG\in \cF_r(n-k+1,S^r_{\ell_k})}\sum_{i=r}^s\binom{k-1}{s-i}\cN(K_i^{(r)},\cG).
\end{eqnarray*}
We will show that $\cH$ contains a copy of $S^r_{\ell_1}\cup\cdots\cup S^r_{\ell_k}$ by induction on $k$. If $k=1$, then $\cN(K_s^{(r)},\cH)> \max_{\cG\in \cF_r(n,S^r_{\ell_1})}\cN(K_s^{(r)},\cG)=\ex_r(n,K_s^{(r)},S^r_{\ell_1})$, and thus $\cH$ contains a copy of $S_{\ell_1}^r$. Suppose $k\geq2$ and the statement holds for any $k'<k$.
By Theorem \ref{thm1}, there exists a constant $c(s,r,\ell_1)$ such that
\begin{eqnarray}
{\rm{ex}}_r(n,K_s^{(r)},S_{\ell_1}^r)\leq c(s,r,\ell_1)n^{r-2}.
\end{eqnarray}

Let $u\in V(\cH)$ be a vertex such that $|\cK_{u}(K_s^{(r)},\cH)|=\max_{v\in V(\cH)}|\cK_{v}(K_s^{(r)},\cH)|$. Now we consider the following two cases.



\smallskip
\noindent {\bf{Case 1.}} $|\cK_{u}(K_s^{(r)},\cH)|< \frac{\cN(K_s^{(r)},\cH)-c(s,r,\ell_1)n^{r-2}}{\sum_{i=2}^{k}\left(\ell_i(r-1)+1\right)}$.
\smallskip
\smallskip

By Theorem \ref{thm1}, for any $\cG\in \cF_r(n-k+1,S^r_{\ell_k})$ and all $r\leq i\leq s$, we have $\cN(K_i^{(r)},\cG)=O(n^{r-2})$. Therefore, $\max_{\cG\in \cF_r(n-k+1,S^r_{\ell_k})}\sum_{i=r}^s\binom{k-1}{s-i}\cN(K_i^{(r)},\cG)=O(n^{r-2})$. Since $s\leq k+r-2$ and $n$ is sufficiently large,
\begin{eqnarray*}
\cN(K_s^{(r)},\cH)&>& \sum_{i=0}^{r-1}\binom{k-1}{s-i}\binom{n-k+1}{i}+\max_{\cG\in \cF_r(n-k+1,S^r_{\ell_k})}\sum_{i=r}^s\binom{k-1}{s-i}\cN(K_i^{(r)},\cG) \\
&\geq& \binom{k-1}{s-r+1}\binom{n-k+1}{r-1} \\
&>& \sum_{i=0}^{r-1}\binom{k-2}{s-i}\binom{n-k+2}{i}+\max_{\cG\in \cF_r(n-k+1,S^r_{\ell_k})}\sum_{i=r}^s\binom{k-1}{s-i}\cN(K_i^{(r)},\cG).
\end{eqnarray*}
By induction hypothesis, $\cH$ contains a copy of $S^r_{\ell_2}\cup\cdots\cup S^r_{\ell_k}$, denoted by $S$. Therefore, $|V(S)|=\sum_{i=2}^{k}\left(\ell_i(r-1)+1\right)$. Then
\begin{eqnarray*}
\cN(K_s^{(r)},\cH-V(S))&>& \cN(K_s^{(r)},\cH)-|V(S)|\cdot \frac{\cN(K_s^{(r)},\cH)-c(s,r,\ell_1)n^{r-2}}{\sum_{i=2}^{k}\left(\ell_i(r-1)+1\right)} \\
&=& c(s,r,\ell_1)n^{r-2},
\end{eqnarray*}
which implies $\cH-V(S)$ contains a copy of $S_{\ell_1}^r$ by (4.1). Thus, $\cH$ contains a copy of $S^r_{\ell_1}\cup\cdots\cup S^r_{\ell_k}$.

\smallskip
\noindent {\bf{Case 2.}} $|\cK_{u}(K_s^{(r)},\cH)|\geq \frac{\cN(K_s^{(r)},\cH)-c(s,r,\ell_1)n^{r-2}}{\sum_{i=2}^{k}\left(\ell_i(r-1)+1\right)}$.
\smallskip
\smallskip

Recall that $\cN(K_s^{(r)},\cH)>\binom{k-1}{s-r+1}\binom{n-k+1}{r-1}$. Therefore, $|\cK_{u}(K_s^{(r)},\cH)|=\Omega(n^{r-1})$. Set
\begin{eqnarray*}
\cK^{s-1}=\big\{K_s^{(r)}-\{u\}\ \big|\ K_s^{(r)}\in \cK_{u}(K_s^{(r)},\cH)\big\}.
\end{eqnarray*}
Then $|\cK^{s-1}|=|\cK_{u}(K_s^{(r)},\cH)|=\Omega(n^{r-1})$. 

\begin{clm}\label{clm3}
$\cH-\{u\}$ is $S^r_{\ell_2}\cup\cdots\cup S^r_{\ell_k}$-free.
\end{clm}
\begin{proof}[\bf Proof of Claim]
Suppose to the contrary that $\cH-\{u\}$ contains a copy of $S^r_{\ell_2}\cup\cdots\cup S^r_{\ell_k}$, denoted by $S'$. Therefore, $|V(S')|=\sum_{i=2}^{k}\left(\ell_i(r-1)+1\right)$. 
Observe that
$$\big|\big\{K_{s-1}^{(r)}\in \cK^{s-1}\,\big|\,|V(K_{s-1}^{(r)})\cap V(S')|\geq s-r+1\big\}\big|=O(n^{r-2}).$$
Then there are at least $|\cK^{s-1}|-O(n^{r-2})=\Omega(n^{r-1})$ copies of $K_{s-1}^{(r)}$ in $\cK^{s-1}$ such that $|V(K_{s-1}^{(r)})\backslash V(S')|\geq r-1$. For $r-1\leq i\leq s-1$, set
\begin{eqnarray*}
X_i=\big\{K_{s-1}^{(r)}\in \cK^{s-1}\,\big|\,|V(K_{s-1}^{(r)})\backslash V(S')|=i\big\}.
\end{eqnarray*}
Then $\sum_{i=r-1}^{s-1}|X_i|=\Omega(n^{r-1})$. Now we define the following collections. Let
\begin{eqnarray*}
X'_{i}=\big\{K_{s-1}^{(r)}-V(S')\,\big|\,K_{s-1}^{(r)}\in X_i\big\}
\end{eqnarray*}
for $r-1\leq i\leq s-1$. Note that $X'_{i}$ is a collection of $(r-1)$-sets in $V(\cH)\backslash (\{u\}\cup V(S'))$ when $i=r-1$, and $X'_{i}$ is a collection of $K_{i}^{(r)}$ in $V(\cH)\backslash (\{u\}\cup V(S'))$ when $r-1<i\leq s-1$.
By the definitions of $X_i$ and $X'_{i}$, we have $|X'_{i}|\geq |X_i|/\binom{|V(S')|}{s-i-1}$ for each $r-1\leq i\leq s-1$. Thus,
\begin{eqnarray*}
\sum_{i=r-1}^{s-1}|X'_i|&\geq& \sum_{i=r-1}^{s-1}|X_i|\cdot \binom{|V(S')|}{s-i-1}^{-1} \\
&\geq& \min_{r-1\leq i\leq s-1}\binom{|V(S')|}{s-i-1}^{-1}\cdot \left(|X_{r-1}|+|X_r|+\cdots+|X_{s-1}|\right) \\
&=& \Omega(n^{r-1}).
\end{eqnarray*}
By the pigeonhole principle, $|X'_i|=\Omega(n^{r-1})$ for some $r-1\leq i\leq s-1$. Note that $i$ is fixed in what follows.

Let $\cS^{r-1}=\big\{K_{r-1}^{(r)}\ \big|\ K_{r-1}^{(r)}\subseteq V(K_{j}^{(r)}) \ \text{and}\ K_{j}^{(r)}\in X'_j,\,r-1\leq j\leq s-1\big\}$. Consider the $(r-1)$-graph $\cH^{r-1}=\left(V(\cH)\backslash (\{u\}\cup V(S')),\cS^{r-1}\right)$. 
We claim that $\cH^{r-1}$ is $M_{\ell_1}^{r-1}$-free. Otherwise, assume that $\cH^{r-1}$ contains a copy of $M_{\ell_1}^{r-1}$. Note that each $K_{j}^{(r)}\in X'_j$ forms a $K_{j+1}^{(r)}$ together with the vertex $u$. This implies that $\cH[\{u\}\cup V(M_{\ell_1}^{r-1})]$ contains a copy of $S_{\ell_1}^{r}$ with $u$ as its center. Thus, $S_{\ell_1}^{r}\cup S'$ forms a copy of $S^r_{\ell_1}\cup\cdots\cup S^r_{\ell_k}$ in $\cH$, and we are done.

Since $\cH^{r-1}$ is $M_{\ell_1}^{r-1}$-free, by Theorem \ref{Liu-Wang-matching}, we have $\cN(K_i^{(r-1)},\cH^{r-1})=O(n^{r-2})$. By the definitions of $\cS^{r-1}$ and $\cH^{r-1}$, it follows that $|X'_{i}|\leq \cN(K_i^{(r-1)},\cH^{r-1})$, since each $K_{i}^{(r)}\in X'_{i}$ corresponds to a copy of $K_{i}^{(r-1)}$ in $\cH^{r-1}$. Thus, $|X'_{i}|=O(n^{r-2})$, which contradicts $|X'_i|=\Omega(n^{r-1})$. This completes the proof of Claim \ref{clm3}.
\end{proof}

By Claim \ref{clm3} and the induction hypothesis, we have
\begin{eqnarray*}
\cN(K_{s-1}^{(r)},\cH-\{u\})&\leq& \sum_{i=0}^{r-1}\binom{k-2}{s-i-1}\binom{n-k+1}{i}+\max_{\cG\in \cF_r(n-k+1,S^r_{\ell_k})}\sum_{i=r}^s\binom{k-2}{s-i-1}\cN(K_i^{(r)},\cG).
\end{eqnarray*}
Therefore,
\begin{eqnarray*}
|\cK_{u}(K_s^{(r)},\cH)|&\leq& \cN(K_{s-1}^{(r)},\cH-\{u\}) \\
&\leq& \sum_{i=0}^{r-1}\binom{k-2}{s-i-1}\binom{n-k+1}{i}+\max_{\cG\in \cF_r(n-k+1,S^r_{\ell_k})}\sum_{i=r}^s\binom{k-2}{s-i-1}\cN(K_i^{(r)},\cG).
\end{eqnarray*}

Thus,
\begin{eqnarray*}
\cN(K_{s}^{(r)},\cH-\{u\})&=& \cN(K_{s}^{(r)},\cH)-|\cK_{u}(K_s^{(r)},\cH)| \\
&>& \sum_{i=0}^{r-1}\binom{k-1}{s-i}\binom{n-k+1}{i}+\max_{\cG\in \cF_r(n-k+1,S^r_{\ell_k})}\sum_{i=r}^s\binom{k-1}{s-i}\cN(K_i^{(r)},\cG) \\
&{}& -\sum_{i=0}^{r-1}\binom{k-2}{s-i-1}\binom{n-k+1}{i}-\max_{\cG\in \cF_r(n-k+1,S^r_{\ell_k})}\sum_{i=r}^s\binom{k-2}{s-i-1}\cN(K_i^{(r)},\cG) \\
&=& \sum_{i=0}^{r-1}\binom{k-2}{s-i}\binom{n-k+1}{i}+\max_{\cG\in \cF_r(n-k+1,S^r_{\ell_k})}\sum_{i=r}^s\binom{k-2}{s-i}\cN(K_i^{(r)},\cG).
\end{eqnarray*}
This implies that $\cH-\{u\}$ contains a copy of $S^r_{\ell_2}\cup\cdots\cup S^r_{\ell_k}$ by the induction hypothesis, which contradicts Claim \ref{clm3}. This completes the proof of \textbf{(i)}.

Now we consider the upper bound in \textbf{(ii)}. Suppose $\cH$ is an $r$-graph on $n$ vertices with $\cN(K_s^{(r)},\cH)=\omega(n^{r-2})$. When $k=1$, the conclusion follows from Theorem \ref{thm1}. Analogous to the proof of the upper bound in \textbf{(i)}, we may show that $\cH$ contains a copy of $S^r_{\ell_1}\cup\cdots\cup S^r_{\ell_k}$ by induction on $k$, completing the proof.

If $k+r-1\leq s\leq \sum_{i=1}^k(\ell_i+1)+r-3$, then we consider the following $r$-graph $\cH_2$: take a vertex set partition $A\cup B$ where $|A|=\sum_{i=1}^k(\ell_i+1)-1$ and $|B|=n-\sum_{i=1}^k(\ell_i+1)+1$, then add as hyperedges all $r$-sets containing at least two vertices from $A$. We claim that $\cH_2$ is $S^r_{\ell_1}\cup\cdots\cup S^r_{\ell_k}$-free. Suppose to the contrary that $\cH_2$ contains a copy of $S^r_{\ell_1}\cup\cdots\cup S^r_{\ell_k}$, denoted by $\cC$. Similar to the analysis in Theorem \ref{thm1}, we have $|V(S_{\ell_i}^r)\cap A|\geq \min\{\ell_i+1,2\ell_i\}=\ell_i+1$. Then $|V(\cC)\cap A|\geq \sum_{i=1}^k(\ell_i+1)$, which contradicts $|A|=\sum_{i=1}^k(\ell_i+1)-1$. Therefore, $\cH_2$ is $S^r_{\ell_1}\cup\cdots\cup S^r_{\ell_k}$-free. Then
$$\ex_r(n,K_s^{(r)},S_\ell^r)\geq \cN(K_s^{(r)},\cH_2)=\sum_{i=0}^{r-2}\binom{\sum_{i=1}^k(\ell_i+1)-1}{s-r+2+i}\binom{n-\sum_{i=1}^k(\ell_i+1)+1}{r-2-i}.$$
Combining the upper bound, we have $\ex_r(n,K_s^{(r)},S^r_{\ell_1}\cup\cdots\cup S^r_{\ell_k})=\Theta(n^{r-2})$.
\end{proof}

\section{\normalsize Proofs of Theorems \ref{thm1.1} and \ref{thm3}}\label{5}
In \cite{ZhY2}, we determined the exact Tur\'{a}n number for the vertex-disjoint union of expansions of paths and cycles. For the sake of completeness, a proof of this result is included in the Appendix.

\begin{thm}[Zhou and Yuan \cite{ZhY2}]\label{Zhou-Yuan}
Let integers $r\geq3$, $k\geq1$, $0\leq p\leq k$ and $\ell_1,\dots,\ell_k\geq3$, where $\ell_i\neq4$ for all $i$ if $r=3$. Then for sufficiently large $n$,
\begin{eqnarray*}
{\rm{ex}}_r(n,P_{\ell_1}^r\cup\cdots\cup P_{\ell_{p}}^r\cup C_{\ell_{p+1}}^r\cup\cdots\cup C_{\ell_{k}}^r)= \binom{n}{r}-\binom{n-\sum_{i=1}^{k}\big\lfloor\frac{\ell_i+1}{2}\big\rfloor+1}{r}+c(n;\ell_1,\dots,\ell_k),
\end{eqnarray*}
where $c(n;\ell_1,\dots,\ell_k)=\binom{n-\sum_{i=1}^{k}\lfloor\frac{\ell_i+1}{2}\rfloor-1}{r-2}$ if all of $\ell_1,\dots,\ell_k$ are even, and $c(n;\ell_1,\dots,\ell_k)=0$ otherwise.
\end{thm}

Recall that $S_{n,t}^r$ denotes the $n$-vertex $r$-graph consisting of all hyperedges that intersect a fixed set of $t$ vertices. Now let us start with the proof of Theorem \ref{thm1.1} that we restate here for convenience.

\begin{thmbis}{thm1.1}
Let $r\geq3$, $k\geq1$ and $\ell_1,\dots,\ell_k\geq 5$ be integers and $t=\sum_{i=1}^k\lfloor\frac{\ell_i+1}{2}\rfloor-1$. Suppose $Q_i\in\{P_{\ell_i}^r,C_{\ell_i}^r\}$, $Q_1\cup\cdots\cup Q_k\in \{P_{\ell_1}^r\cup\cdots\cup P_{\ell_k}^r, C_{\ell_1}^r\cup\cdots\cup C_{\ell_k}^r\}$ and $n$ is sufficiently large.

\textbf{(i)} If $s \leq t+r-1$, then
\begin{eqnarray*}
{\rm{ex}}_r(n,K_s^{(r)},Q_1\cup\cdots\cup Q_k)= (1+o(1))\cN(K_s^{(r)},S_{n,t}^r)=\binom{t}{s-r+1}\frac{n^{r-1}}{(r-1)!}+o(n^{r-1}).
\end{eqnarray*}

\textbf{(ii)} If $s \geq t+r$, then
\begin{eqnarray*}
{\rm{ex}}_r(n,K_s^{(r)},Q_1\cup\cdots\cup Q_k)= o(n^{r-1}).
\end{eqnarray*}
\end{thmbis}

\begin{proof}[{\bf{Proof.}}]
For the lower bound, we consider the $r$-graph $S_{n,t}^r$. Observe that any $P_{\ell_i}^r$ or $C_{\ell_i}^r$ in $S_{n,t}^r$ contains at least $\big\lfloor\frac{\ell_i+1}{2}\big\rfloor$ vertices in the $t$-set of $S_{n,t}^r$. Thus, $S_{n,t}^r$ is $Q_1\cup\cdots\cup Q_k$-free as $t<\sum_{i=1}^k\lfloor\frac{\ell_i+1}{2}\rfloor$.

We now prove the upper bounds in \textbf{(i)} and \textbf{(ii)} together, proceeding by double induction on $s$ and $k$. 
For the case when $k=1$ and $s\geq r$, the conclusion follows from Theorem \ref{Gerbner-path}. For the case when $s=r$ and $k\geq1$, the conclusion follows from Theorems \ref{Bushaw-Kettle} (when $Q_1\cup\cdots\cup Q_k=P_{\ell_1}^r\cup\cdots\cup P_{\ell_k}^r$) and \ref{Zhou-Yuan} (when $Q_1\cup\cdots\cup Q_k=C_{\ell_1}^r\cup\cdots\cup C_{\ell_k}^r$).
Now we assume that $s>r$, $k>1$ and the conclusion holds for any $s'\leq s$ and $k-1$. Let $\cH$ be an $r$-graph on $n$ vertices with
\begin{eqnarray*}
\cN(K_s^{(r)},\cH)> \binom{t}{s-r+1}\frac{n^{r-1}}{(r-1)!}+o(n^{r-1}).
\end{eqnarray*}
This implies that $\cN(K_s^{(r)},\cH)\sim an^{r-1}$ as $n$ is sufficiently large, where $a>\binom{t}{s-r+1}\frac{1}{(r-1)!}$ if $s\leq t+r-1$, and $a>0$ if $s \geq t+r$.
Next, we will show that $\cH$ contains a copy of $Q_1\cup\cdots\cup Q_k$. 
For convenience, let $\ell'=\lfloor\frac{\ell_1+1}{2}\big\rfloor$. Since $\binom{t}{s-r+1}\geq\binom{\ell'-1}{s-r+1}$ and $n$ is sufficiently large,
\begin{eqnarray*}
\cN(K_s^{(r)},\cH)\sim an^{r-1}> \binom{t}{s-r+1}\frac{n^{r-1}}{(r-1)!}+o(n^{r-1})\geq \binom{\ell'-1}{s-r+1}\frac{n^{r-1}}{(r-1)!}+o(n^{r-1}),
\end{eqnarray*}
which implies $\cH$ contains a copy of $Q_1$ by Theorem \ref{Gerbner-path}, denoted by $F$. Therefore, $|V(F)|=\ell_1(r-1)+1$ if $Q_1=P_{\ell_1}^r$, and $|V(F)|=\ell_1(r-1)$ if $Q_1=C_{\ell_1}^r$.

We may assume that
\begin{eqnarray*}
\cN\big(K_s^{(r)},\cH-V(F)\big)\leq \binom{t-\ell'}{s-r+1}\frac{n^{r-1}}{(r-1)!}+o(n^{r-1}).
\end{eqnarray*}
Otherwise, by the induction hypothesis, we may find a copy of $Q_2\cup\cdots\cup Q_k$ in $\cH-V(F)$. Thus, $\cH$ contains a copy of $Q_1\cup\cdots\cup Q_k$ and we are done.

Let $\cK_{V(F)}(K_s^{(r)},\cH)=\{K_s^{(r)}\subseteq \cH\,|\,V(K_s^{(r)})\cap V(F)\neq \emptyset\}$, namely, the collection of all $K_s^{(r)}$ in $\cH$ incident to vertices in $V(F)$. Then
\begin{eqnarray}
|\cK_{V(F)}(K_s^{(r)},\cH)|&=& \cN(K_s^{(r)},\cH)-\cN(K_s^{(r)},\cH-V(F)) \notag \\
&>& \left(\binom{t}{s-r+1}-\binom{t-\ell'}{s-r+1}\right)\frac{n^{r-1}}{(r-1)!}+o(n^{r-1}).
\end{eqnarray}

Now we estimate the upper bound for $|\cK_{V(F)}(K_s^{(r)},\cH)|$. We partition $\cK_{V(F)}(K_s^{(r)},\cH)$ into the following $s-r+2$ parts. Let
\begin{eqnarray*}
&{}&\cK^{*}_{V(F)}(K_s^{(r)},\cH)=\big\{K_s^{(r)}\in \cK_{V(F)}(K_s^{(r)},\cH)\ \big|\ |V(K_s^{(r)})\cap V(F)|\geq s-r+2\big\}, \\
&{}&\cK^i_{V(F)}(K_s^{(r)},\cH)=\big\{K_s^{(r)}\in \cK_{V(F)}(K_s^{(r)},\cH)\ \big|\ |V(K_s^{(r)})\cap V(F)|=i\big\},
\end{eqnarray*}
for $1\leq i\leq s-r+1$. Clearly, $|\cK^{*}_{V(F)}(K_s^{(r)},\cH)|=O(n^{r-2})$ by the definition of $\cK^{*}_{V(F)}(K_s^{(r)},\cH)$. For $1\leq i\leq s-r+1$, set
\begin{eqnarray*}
\cK^i(K_{s-i}^{(r)},\cH)=\big\{K_s^{(r)}-V(F)\ \big|\ K_s^{(r)}\in \cK^i_{V(F)}(K_s^{(r)},\cH)\big\}.
\end{eqnarray*}
Note that $\cK^{i}(K_{s-i}^{(r)},\cH)$ is a collection of $(r-1)$-sets in $V(\cH)\backslash V(F)$ when $i=s-r+1$, and $\cK^i(K_{s-i}^{(r)},\cH)$ is a collection of $K_{s-i}^{(r)}$ in $V(\cH)\backslash V(F)$ when $i<s-r+1$.
For $1\leq i\leq s-r+1$ and $K_{s-i}^{(r)}\in \cK^i(K_{s-i}^{(r)},\cH)$, define the \textit{neighborhood} of $K_{s-i}^{(r)}$ in $V(F)$ as a vertex set as follows:
$$N_{V(F)}(K_{s-i}^{(r)})=\big\{\bigcup V(K_{i}^{(r)})\ \big|\ V(K_{i}^{(r)})\subseteq V(F) \ \text{and}\ \cH[V(K_{i}^{(r)})\cup V(K_{s-i}^{(r)})]=K_{s}^{(r)}\big\}.$$

We now partition $\cK^i(K_{s-i}^{(r)},\cH)$ into two parts for each $1\leq i\leq s-r+1$. Set
\begin{eqnarray*}
&{}&X_{s-i}=\big\{K_{s-i}^{(r)}\in\cK^i(K_{s-i}^{(r)},\cH)\ \big|\ |N_{V(F)}(K_{s-i}^{(r)})|\geq \ell'\big\}, \\
&{}&Y_{s-i}=\big\{K_{s-i}^{(r)}\in\cK^i(K_{s-i}^{(r)},\cH)\ \big|\ |N_{V(F)}(K_{s-i}^{(r)})|<\ell'\big\}.
\end{eqnarray*}
Then
\begin{eqnarray*}
|\cK^i_{V(F)}(K_s^{(r)},\cH)|\leq |X_{s-i}|\cdot \binom{|V(F)|}{i}+|Y_{s-i}|\cdot \binom{\ell'-1}{i}
\end{eqnarray*}
for each $1\leq i\leq s-r+1$. It is easy to see that $|Y_{r-1}|\leq \binom{n-|V(F)|}{r-1}$ and $|Y_{s-i}|\leq \cN(K_{s-i}^{(r)},\cH-V(F))$ for each $i<s-r+1$.
Since $\cH-V(F)$ is $Q_2\cup\cdots\cup Q_k$-free, by the induction hypothesis, we have
\begin{eqnarray*}
\cN\big(K_{s-i}^{(r)},\cH-V(F)\big)\leq \binom{t-\ell'}{s-i-r+1}\frac{n^{r-1}}{(r-1)!}+o(n^{r-1})
\end{eqnarray*}
for each $1\leq i<s-r+1$. Therefore,
\begin{eqnarray*}
|\cK^{s-r+1}_{V(F)}(K_s^{(r)},\cH)|\leq |X_{r-1}|\cdot \binom{|V(F)|}{s-r+1}+\binom{\ell'-1}{s-r+1}\cdot \binom{n-|V(F)|}{r-1}
\end{eqnarray*}
and
\begin{eqnarray*}
|\cK^i_{V(F)}(K_s^{(r)},\cH)|\leq |X_{s-i}|\cdot \binom{|V(F)|}{i}+\binom{\ell'-1}{i}\cdot \binom{t-\ell'}{s-i-r+1}\frac{n^{r-1}}{(r-1)!}+o(n^{r-1})
\end{eqnarray*}
for each $1\leq i<s-r+1$. Thus,
\begin{eqnarray}
|\cK_{V(F)}(K_s^{(r)},\cH)|&=& \sum_{i=1}^{s-r+1}|\cK^i_{V(F)}(K_s^{(r)},\cH)|+|\cK^{*}_{V(F)}(K_s^{(r)},\cH)| \notag\\
&\leq& \left[\sum_{i=1}^{s-r+1}|X_{s-i}|\cdot \binom{|V(F)|}{i}\right]+\Bigg[\binom{\ell'-1}{s-r+1}\cdot \binom{n-|V(F)|}{r-1} \notag\\
&{}& +\sum_{i=1}^{s-r}\binom{\ell'-1}{i}\cdot \binom{t-\ell'}{s-i-r+1}\frac{n^{r-1}}{(r-1)!}\Bigg]+o(n^{r-1}) \notag\\
&\leq& \max_{1\leq i\leq s-r+1}\binom{|V(F)|}{i}\cdot \left(|X_{s-1}|+|X_{s-2}|+\cdots+|X_{r-1}|\right) \notag\\
&{}& +\left(\binom{\ell'-1}{s-r+1}+\sum_{i=1}^{s-r}\binom{\ell'-1}{i}\cdot \binom{t-\ell'}{s-i-r+1}\right)\frac{n^{r-1}}{(r-1)!}+o(n^{r-1}) \notag\\
&=& \max_{1\leq i\leq s-r+1}\binom{|V(F)|}{i}\cdot \left(|X_{r-1}|+|X_{r}|+\cdots+|X_{s-1}|\right) \notag\\
&{}& +\left(\binom{t-1}{s-r+1}-\binom{t-\ell'}{s-r+1}\right)\frac{n^{r-1}}{(r-1)!}+o(n^{r-1}),
\end{eqnarray}
where the last equality holds by Vandermonde's identity.

Combining (5.1) and (5.2), we get
\begin{eqnarray}
\left(|X_{r-1}|+|X_{r}|+\cdots+|X_{s-1}|\right)> \frac{\binom{t-1}{s-r}}{\max_{1\leq i\leq s-r+1}\binom{|V(F)|}{i}}\cdot\frac{n^{r-1}}{(r-1)!}+o(n^{r-1}).
\end{eqnarray}
By the definition of $X_i$ ($r-1\leq i\leq s-1$), we have $|N_{V(F)}(K_{i}^{(r)})|\geq \ell'$ for any $K_i^{(r)}\in X_i$. Therefore, by (5.3) and the pigeonhole principle, there exists an $\ell'$-set $U\subseteq V(F)$ such that, there are at least
\begin{eqnarray*}
\frac{1}{\binom{|V(F)|}{\ell'}}\cdot\frac{\binom{t-1}{s-r}}{\max_{1\leq i\leq s-r+1}\binom{|V(F)|}{i}}\cdot\frac{n^{r-1}}{(r-1)!}+o(n^{r-1})=\Omega(n^{r-1})
\end{eqnarray*}
copies of $K_{r-1}^{(r)},K_{r}^{(r)},\dots,K_{s-1}^{(r)}\in \bigcup_{i=r-1}^{s-1}X_i$ such that each $K_{i}^{(r)}$ forms a $K_{i+1}^{(r)}$ together with any vertex in $U$ (note that it is possible that $K_{i}^{(r)}\subseteq K_{j}^{(r)}$ for $r-1\leq i<j\leq s-1$). Without loss of generality, let $X'_i\subseteq X_i$ denote the collection of such $K_{i}^{(r)}$ for each $r-1\leq i\leq s-1$, respectively. Then $\sum_{i=r-1}^{s-1}|X'_i|=\Omega(n^{r-1})$. By the pigeonhole principle, $|X'_i|=\Omega(n^{r-1})$ for some $r-1\leq i\leq s-1$. Note that $i$ is fixed in what follows.

\begin{clm}\label{clm2}
$\cH-U$ is $Q_2\cup\cdots\cup Q_k$-free.
\end{clm}
\begin{proof}[\bf Proof of Claim]
Suppose to the contrary that $\cH-U$ contains a copy of $Q_2\cup\cdots\cup Q_k$, denoted by $F'$. Therefore, $|V(F')|=\sum_{i=2}^k\left(\ell_i(r-1)+1\right)$ if $Q_2\cup\cdots\cup Q_k=P_{\ell_2}^r\cup\cdots\cup P_{\ell_k}^r$, and $|V(F')|=\sum_{i=2}^k\ell_i(r-1)$ if $Q_2\cup\cdots\cup Q_k=C_{\ell_2}^r\cup\cdots\cup C_{\ell_k}^r$.
Moreover, $|V(F')\backslash V(F)|\leq \sum_{i=2}^k\left(\ell_i(r-1)+1\right)$.

Observe that
$$\big|\big\{K_i^{(r)}\in X'_i\,\big|\,|V(K_i^{(r)})\cap V(F')|\geq i-r+2\big\}\big|=O(n^{r-2}).$$
Then there are at least $|X'_i|-O(n^{r-2})=\Omega(n^{r-1})$ copies of $K_i^{(r)}$ in $X'_i$ such that $|V(K_i^{(r)})\backslash V(F')|\geq r-1$. For $r-1\leq j\leq i$, set
\begin{eqnarray*}
X'_{i,j}=\big\{K_i^{(r)}\in X'_i\,\big|\,|V(K_i^{(r)})\backslash V(F')|=j\big\}.
\end{eqnarray*}
Then $\sum_{j=r-1}^i|X'_{i,j}|=\Omega(n^{r-1})$. Now we define the following collections. Let
\begin{eqnarray*}
X''_{j}=\big\{K_i^{(r)}-V(F')\,\big|\,K_i^{(r)}\in X'_{i,j}\big\}
\end{eqnarray*}
for $r-1\leq j\leq i$. Note that $X''_{j}$ is a collection of $(r-1)$-sets in $V(\cH)\backslash (V(F)\cup V(F'))$ when $j=r-1$, and $X''_{j}$ is a collection of $K_{j}^{(r)}$ in $V(\cH)\backslash (V(F)\cup V(F'))$ when $r-1<j\leq i$.
By the definitions of $X'_{i,j}$ and $X''_{j}$, we have $|X''_{j}|\geq |X'_{i,j}|/\binom{|V(F')|}{i-j}$ for each $r-1\leq j\leq i$. Thus,
\begin{eqnarray*}
\sum_{j=r-1}^i|X''_{j}|&\geq& \sum_{j=r-1}^i|X'_{i,j}|\cdot \binom{|V(F')|}{i-j}^{-1} \\
&\geq& \min_{r-1\leq j\leq i}\binom{|V(F')|}{i-j}^{-1}\cdot \left(|X'_{i,r-1}|+|X'_{i,r}|+\cdots+|X'_{i,i}|\right) \\
&=& \Omega(n^{r-1}).
\end{eqnarray*}
By the pigeonhole principle, $|X''_j|=\Omega(n^{r-1})$ for some $r-1\leq j\leq i$. Note that $j$ is also fixed in what follows.

Let $\cS^{r-1}=\big\{K_{r-1}^{(r)}\ \big|\ K_{r-1}^{(r)}\subseteq V(K_{m}^{(r)}) \ \text{and}\ K_{m}^{(r)}\in X''_m,\,r-1\leq m\leq i\big\}$. Consider the $(r-1)$-graph $\cH^{r-1}=\left(V(\cH)\backslash (V(F)\cup V(F')),\cS^{r-1}\right)$. 
We claim that $\cH^{r-1}$ is $P_{3\ell'+2}^{r-1}$-free. Otherwise, assume that $\cH^{r-1}$ contains a copy of $P_{3\ell'+2}^{r-1}$. We will find a copy of $Q_1$ in $\cH$ that is vertex-disjoint from $F'$. Thus, $Q_1\cup F'$ forms a copy of $Q_1\cup\cdots\cup Q_k$ in $\cH$, and we are done.
Let $U=\{v_1,\dots,v_{\ell'}\}$. Recall that each $K_{i}^{(r)}\in X'_i$ forms a $K_{i+1}^{(r)}$ together with any vertex in $U$. This implies that for $r-1\leq m\leq i$, each $K_{m}^{(r)}\in X''_m$ forms a $K_{m+1}^{(r)}$ together with any vertex in $U$. We consider the following two cases.

\smallskip
\noindent {\bf{Case 1.}} Either $\ell_1$ is even, or $Q_1=P^r_{\ell_1}$.
\smallskip
\smallskip

Since $(\ell'+1)S_2^{r-1}\subseteq P_{3\ell'+2}^{r-1}$, $\cH^{r-1}$ contains a copy of $(\ell'+1)S_2^{r-1}$. We denote it and its hyperedges by $S^{(1)}\cup\cdots\cup S^{(\ell'+1)}$ and $e^{(1)}_1,e^{(1)}_2,\dots,e^{(\ell'+1)}_1,e^{(\ell'+1)}_2$, where $e^{(i)}_1,e^{(i)}_2\in E(S^{(i)})$. Then we may construct a copy of $P^r_{2\ell'}$ or a copy of $C^r_{2\ell'}$ that is vertex-disjoint from $F'$ as follows:
$$P^r_{2\ell'}:\ \{v_1\}\cup e_2^{(1)},\{v_1\}\cup e_1^{(2)},\{v_2\}\cup e_2^{(2)},\dots,\{v_{\ell'}\}\cup e_2^{({\ell'})},\{v_{\ell'}\}\cup e_1^{({\ell'}+1)},$$
$$C^r_{2\ell'}:\ \{v_1\}\cup e_1^{(1)},\{v_2\}\cup e_2^{(1)},\{v_2\}\cup e_1^{(2)},\dots,\{v_{\ell'}\}\cup e_1^{({\ell'})},\{v_{1}\}\cup e_2^{({\ell'})}.$$

Observe that $P^r_{\ell_1}\subseteq P^r_{2\ell'}$, and $C^r_{2\ell'}=C^r_{\ell_1}$ when $\ell_1$ is even. Thus, we may find a copy of $Q_1$ in $\cH$ that is vertex-disjoint from $F'$, and we are done.

\smallskip
\noindent {\bf{Case 2.}} $\ell_1$ is odd and $Q_1=C^r_{\ell_1}$.
\smallskip
\smallskip

Since $(\ell'-2)S_2^{r-1}\cup P_3^{r-1}\subseteq P_{3\ell'+2}^{r-1}$, $\cH^{r-1}$ contains a copy of $(\ell'-2)S_2^{r-1}\cup P_3^{r-1}$. We denote it and its hyperedges by $S^{(1)}\cup\cdots\cup S^{(\ell'-2)}\cup P$ and $e^{(1)}_1,e^{(1)}_2,\dots,e^{(\ell'-2)}_1,e^{(\ell'-2)}_2,e_1,e_2,e_3$, where $e^{(i)}_1,e^{(i)}_2\in E(S^{(i)})$ and $e_1,e_2,e_3\in E(P)$. Then we may construct a copy of $C^r_{\ell_1}$ that is vertex-disjoint from $F'$ as follows:
$$C^r_{\ell_1}:\ \{v_1\}\cup e_1^{(1)},\{v_2\}\cup e_2^{(1)},\dots,\{v_{\ell'-1}\}\cup e_2^{({\ell'-2})},\{v_{\ell'-1}\}\cup e_1,\{v_{\ell'}\}\cup e_2,\{v_{1}\}\cup e_3.$$

Thus, $\cH^{r-1}$ is $P_{3\ell'+2}^{r-1}$-free. By Theorems \ref{CC} and \ref{Gerbner-path}, we have $\cN(K_j^{(r-1)},\cH^{r-1})=O(n^{r-2})$. By the definitions of $\cS^{r-1}$ and $\cH^{r-1}$, it follows that $|X''_{j}|\leq \cN(K_j^{(r-1)},\cH^{r-1})$, since each $K_{j}^{(r)}\in X''_{j}$ corresponds to a copy of $K_{j}^{(r-1)}$ in $\cH^{r-1}$. Thus, $|X''_{j}|=O(n^{r-2})$, which contradicts $|X''_j|=\Omega(n^{r-1})$. This completes the proof of Claim \ref{clm2}.
\end{proof}

Let $\cK_{U}(K_s^{(r)},\cH)=\{K_s^{(r)}\subseteq \cH\,|\,V(K_s^{(r)})\cap U\neq \emptyset\}$, namely, the collection of all $K_s^{(r)}$ in $\cH$ incident to vertices in $U$.
Now we estimate the upper bound for $|\cK_{U}(K_s^{(r)},\cH)|$. We partition $\cK_{U}(K_s^{(r)},\cH)$ into the following $s-r+2$ parts. Let
\begin{eqnarray*}
&{}&\cK^{*}_{U}(K_s^{(r)},\cH)=\big\{K_s^{(r)}\in \cK_{U}(K_s^{(r)},\cH)\ \big|\ |V(K_s^{(r)})\cap U|\geq s-r+2\big\}, \\
&{}&\cK^i_{U}(K_s^{(r)},\cH)=\big\{K_s^{(r)}\in \cK_{U}(K_s^{(r)},\cH)\ \big|\ |V(K_s^{(r)})\cap U|=i\big\},
\end{eqnarray*}
for $1\leq i\leq s-r+1$. Clearly, $|\cK^{*}_{U}(K_s^{(r)},\cH)|=O(n^{r-2})$ by the definition of $\cK^{*}_{U}(K_s^{(r)},\cH)$.

Observe that
\begin{eqnarray*}
&{}&\cK^{s-r+1}_{U}(K_s^{(r)},\cH)\leq \binom{|U|}{s-r+1}\cdot \binom{n-|U|}{r-1}=\binom{\ell'}{s-r+1}\cdot \binom{n-\ell'}{r-1}
\end{eqnarray*}
and
\begin{eqnarray*}
&{}&\cK^i_{U}(K_s^{(r)},\cH)\leq \binom{|U|}{i}\cdot \cN(K_{s-i}^{(r)},\cH-U)=\binom{\ell'}{i}\cdot \cN(K_{s-i}^{(r)},\cH-U)
\end{eqnarray*}
for $1\leq i< s-r+1$.

By Claim \ref{clm2} and the induction hypothesis, we have
\begin{eqnarray*}
\cN(K_{s-i}^{(r)},\cH-U)\leq \binom{t-\ell'}{s-i-r+1}\cdot\frac{n^{r-1}}{(r-1)!}+o(n^{r-1})
\end{eqnarray*}
for $1\leq i\leq s-r$.
Therefore,
\begin{eqnarray*}
|\cK_{U}(K_s^{(r)},\cH)|&=& \sum_{i=1}^{s-r+1}|\cK^i_{U}(K_s^{(r)},\cH)|+|\cK^{*}_{U}(K_s^{(r)},\cH)| \\
&\leq& \sum_{i=1}^{s-r}\binom{\ell'}{i}\cdot \cN(K_{s-i}^{(r)},\cH-U)+\binom{\ell'}{s-r+1}\cdot \binom{n-\ell'}{r-1}+O(n^{r-2}) \\
&\leq& \sum_{i=1}^{s-r}\binom{\ell'}{i}\cdot \left(\binom{t-\ell'}{s-i-r+1}\cdot\frac{n^{r-1}}{(r-1)!}+o(n^{r-1})\right)+\binom{\ell'}{s-r+1}\cdot \binom{n-\ell'}{r-1}+O(n^{r-2}) \\
&=& \left(\sum_{i=1}^{s-r}\binom{\ell'}{i}\cdot \binom{t-\ell'}{s-i-r+1}+\binom{\ell'}{s-r+1}\right)\frac{n^{r-1}}{(r-1)!}+o(n^{r-1}) \\
&=& \left(\binom{t}{s-r+1}-\binom{t-\ell'}{s-r+1}\right)\frac{n^{r-1}}{(r-1)!}+o(n^{r-1}),
\end{eqnarray*}
where the last equality holds by Vandermonde's identity.

Thus,
\begin{eqnarray*}
\cN(K_{s}^{(r)},\cH-U)&=& \cN(K_{s}^{(r)},\cH)-|\cK_{U}(K_s^{(r)},\cH)| \\
&>& \binom{t}{s-r+1}\frac{n^{r-1}}{(r-1)!}-\left(\binom{t}{s-r+1}-\binom{t-\ell'}{s-r+1}\right)\frac{n^{r-1}}{(r-1)!}+o(n^{r-1}) \\
&=& \binom{t-\ell'}{s-r+1}\frac{n^{r-1}}{(r-1)!}+o(n^{r-1}).
\end{eqnarray*}
This implies that $\cH-U$ contains a copy of $Q_2\cup\cdots\cup Q_k$ by the induction hypothesis, which contradicts Claim \ref{clm2}.
This completes the proof.
\end{proof}

We now begin with the proof of Theorem \ref{thm3} that we restate here for convenience.

\begin{thmbis}{thm3}
Let integers $r\geq3$, $k\geq1$, $p\geq0$, $\ell_1,\dots,\ell_k\geq 5$, $t_1,\dots,t_{p}\geq2$ if $p\geq1$, and $t=\sum_{i=1}^k\lfloor\frac{\ell_i+1}{2}\rfloor+p-1$. Suppose $Q_i\in\{P_{\ell_i}^r,C_{\ell_i}^r\}$, $Q_1\cup\cdots\cup Q_k\in \{P_{\ell_1}^r\cup\cdots\cup P_{\ell_k}^r, C_{\ell_1}^r\cup\cdots\cup C_{\ell_k}^r\}$ and $n$ is sufficiently large.

\textbf{(i)} If $s \leq t+r-1$, then
\begin{eqnarray*}
{\rm{ex}}_r(n,K_s^{(r)},Q_1\cup\cdots\cup Q_k\cup S^r_{t_1}\cup\cdots\cup S^r_{t_p})= (1+o(1))\cN(K_s^{(r)},S_{n,t}^r)=\binom{t}{s-r+1}\frac{n^{r-1}}{(r-1)!}+o(n^{r-1}).
\end{eqnarray*}

\textbf{(ii)} If $s \geq t+r$, then
\begin{eqnarray*}
{\rm{ex}}_r(n,K_s^{(r)},Q_1\cup\cdots\cup Q_k\cup S^r_{t_1}\cup\cdots\cup S^r_{t_p})= o(n^{r-1}).
\end{eqnarray*}
\end{thmbis}

\begin{proof}[{\bf{Proof.}}]
For the lower bound, we consider the $r$-graph $S_{n,t}^r$. Observe that any $P_{\ell_i}^r$ or $C_{\ell_i}^r$ in $S_{n,t}^r$ contains at least $\big\lfloor\frac{\ell_i+1}{2}\big\rfloor$ vertices in the $t$-set of $S_{n,t}^r$, and any $S_{t_i}^r$ in $S_{n,t}^r$ contains at least one vertex in the $t$-set of $S_{n,t}^r$. Thus, $S_{n,t}^r$ is $Q_1\cup\cdots\cup Q_k\cup S^r_{t_1}\cup\cdots\cup S^r_{t_p}$-free as $t<\sum_{i=1}^k\lfloor\frac{\ell_i+1}{2}\rfloor+p$.

We now prove the upper bounds in \textbf{(i)} and \textbf{(ii)} together. 
Suppose $\cH$ is an $r$-graph on $n$ vertices with
\begin{eqnarray*}
\cN(K_s^{(r)},\cH)> \binom{t}{s-r+1}\frac{n^{r-1}}{(r-1)!}+o(n^{r-1}).
\end{eqnarray*}
This implies that $\cN(K_s^{(r)},\cH)\sim an^{r-1}$ as $n$ is sufficiently large, where $a>\binom{t}{s-r+1}\frac{1}{(r-1)!}$ if $s\leq t+r-1$, and $a>0$ if $s \geq t+r$. Next, we will show that $\cH$ contains a copy of $Q_1\cup\cdots\cup Q_k\cup S^r_{t_1}\cup\cdots\cup S^r_{t_p}$ by induction on $p$. If $p=0$, then the conclusion follows from Theorem \ref{thm1.1}. Suppose $p\geq1$ and the statement holds for any $p'<p$.
By Theorem \ref{thm1}, there exists a constant $c(s,r,t_p)$ such that
\begin{eqnarray}
{\rm{ex}}_r(n,K_s^{(r)},S_{t_p}^r)\leq c(s,r,t_p)n^{r-2}.
\end{eqnarray}

Let $u\in V(\cH)$ be a vertex such that $|\cK_{u}(K_s^{(r)},\cH)|=\max_{v\in V(\cH)}|\cK_{v}(K_s^{(r)},\cH)|$. Assume that $|\cK_{u}(K_s^{(r)},\cH)|< \frac{\cN(K_s^{(r)},\cH)-c(s,r,t_p)n^{r-2}}{\sum_{i=1}^{k}\left(\ell_i(r-1)+1\right)+\sum_{i=1}^{p-1}\left(t_i(r-1)+1\right)}$. Since $n$ is sufficiently large, we have $\cN(K_s^{(r)},\cH)\sim an^{r-1}>\binom{t-1}{s-r+1}\frac{n^{r-1}}{(r-1)!}+o(n^{r-1})$. By induction hypothesis, $\cH$ contains a copy of $Q_1\cup\cdots\cup Q_k\cup S^r_{t_1}\cup\cdots\cup S^r_{t_{p-1}}$, denoted by $Q$. Therefore, $|V(Q)|=\sum_{i=1}^{k}\left(\ell_i(r-1)+1\right)+\sum_{i=1}^{p-1}\left(t_i(r-1)+1\right)$. Then
\begin{eqnarray*}
\cN(K_s^{(r)},\cH-V(Q))&>& \cN(K_s^{(r)},\cH)-|V(Q)|\cdot \frac{\cN(K_s^{(r)},\cH)-c(s,r,t_p)n^{r-2}}{\sum_{i=1}^{k}\left(\ell_i(r-1)+1\right)+\sum_{i=1}^{p-1}\left(t_i(r-1)+1\right)} \\
&=& c(s,r,t_p)n^{r-2},
\end{eqnarray*}
which implies $\cH-V(Q)$ contains a copy of $S_{t_p}^r$ by (5.4). Thus, $\cH$ contains a copy of $Q_1\cup\cdots\cup Q_k\cup S^r_{t_1}\cup\cdots\cup S^r_{t_p}$, and we are done.

Now suppose $|\cK_{u}(K_s^{(r)},\cH)|\geq \frac{\cN(K_s^{(r)},\cH)-c(s,r,t_p)n^{r-2}}{\sum_{i=1}^{k}\left(\ell_i(r-1)+1\right)+\sum_{i=1}^{p-1}\left(t_i(r-1)+1\right)}$. Let
\begin{eqnarray*}
\cK^{s-1}=\big\{K_s^{(r)}-\{u\}\ \big|\ K_s^{(r)}\in \cK_{u}(K_s^{(r)},\cH)\big\}.
\end{eqnarray*}
Then $|\cK^{s-1}|=|\cK_{u}(K_s^{(r)},\cH)|=\Omega(n^{r-1})$. Similar to the proof of Claim \ref{clm3}, we may claim that $\cH-\{u\}$ is $Q_1\cup\cdots\cup Q_k\cup S^r_{t_1}\cup\cdots\cup S^r_{t_{p-1}}$-free. By the induction hypothesis,
\begin{eqnarray*}
\cN(K_{s-1}^{(r)},\cH-\{u\})\leq \binom{t-1}{s-r}\frac{n^{r-1}}{(r-1)!}+o(n^{r-1}).
\end{eqnarray*}
Therefore,
\begin{eqnarray*}
|\cK_{u}(K_s^{(r)},\cH)|\leq \cN(K_{s-1}^{(r)},\cH-\{u\})\leq \binom{t-1}{s-r}\frac{n^{r-1}}{(r-1)!}+o(n^{r-1}).
\end{eqnarray*}

Thus,
\begin{eqnarray*}
\cN(K_{s}^{(r)},\cH-\{u\})&=& \cN(K_{s}^{(r)},\cH)-|\cK_{u}(K_s^{(r)},\cH)| \\
&>& \binom{t}{s-r+1}\frac{n^{r-1}}{(r-1)!}-\binom{t-1}{s-r}\frac{n^{r-1}}{(r-1)!}+o(n^{r-1}) \\
&=& \binom{t-1}{s-r+1}\frac{n^{r-1}}{(r-1)!}+o(n^{r-1}).
\end{eqnarray*}
This implies that $\cH-\{u\}$ contains a copy of $Q_1\cup\cdots\cup Q_k\cup S^r_{t_1}\cup\cdots\cup S^r_{t_{p-1}}$ by the induction hypothesis, which is a contradiction. This completes the proof.
\end{proof}

\section{\normalsize Concluding remarks}\label{6}
\subsection{\normalsize Summary of main results}
In this work, we initiate a systematic study of generalized Tur\'{a}n problems for expansions. Our study considers both non-degenerate and degenerate settings and obtains several exact and asymptotic results.
More specifically, in the non-degenerate case, we show that generalized Tur\'{a}n problems for expansions of graphs $F$ are non-degenerate if and only if $\chi(F)>s$ (see Proposition \ref{prop2.1}). 
Moreover, we determine the exact generalized Tur\'{a}n numbers for expansions of complete graphs (see Theorem \ref{thm0}), and further establish the asymptotics of the generalized Tur\'{a}n number for expansions of the vertex-disjoint union of complete graphs (see Theorem \ref{thm00}).
In the degenerate case, we determine the order of magnitude of $\ex_r(n,K_s^{(r)},S_\ell^r)$ (see Theorem \ref{thm1}), and establish the asymptotics of generalized Tur\'{a}n numbers for expansions of several classes of forests, including star forests (see Theorem \ref{thm2}), linear forests (see Theorem \ref{thm1.1}) and star-path forests (see Theorem \ref{thm3}).

\smallskip
\noindent $\bullet$ In Theorem \ref{thm1.1}, we consider the cases where $Q_1\cup\cdots\cup Q_k\in \{P_{\ell_1}^r\cup\cdots\cup P_{\ell_k}^r, C_{\ell_1}^r\cup\cdots\cup C_{\ell_k}^r\}$ with $Q_i\in\{P_{\ell_i}^r,C_{\ell_i}^r\}$ for each $i$. By a slight modification of the proof of Theorem \ref{thm1.1}, we may obtain the following more general result on expansions of vertex-disjoint unions of paths and cycles.

\begin{thm}\label{thm5.1}
Let $r\geq3$, $k\geq1$ and $\ell_1,\dots,\ell_k\geq 5$ be integers and $t=\sum_{i=1}^k\lfloor\frac{\ell_i+1}{2}\rfloor-1$. Suppose $Q_i\in\{P_{\ell_i}^r,C_{\ell_i}^r\}$ and $n$ is sufficiently large.

\textbf{(i)} If $s \leq t+r-1$, then
\begin{eqnarray*}
{\rm{ex}}_r(n,K_s^{(r)},Q_1\cup\cdots\cup Q_k)= (1+o(1))\cN(K_s^{(r)},S_{n,t}^r)=\binom{t}{s-r+1}\frac{n^{r-1}}{(r-1)!}+o(n^{r-1}).
\end{eqnarray*}

\textbf{(ii)} If $s \geq t+r$, then
\begin{eqnarray*}
{\rm{ex}}_r(n,K_s^{(r)},Q_1\cup\cdots\cup Q_k)= o(n^{r-1}).
\end{eqnarray*}
\end{thm}

\smallskip
\noindent $\bullet$ The proof of Theorem \ref{thm3} proceeds by induction on $p$, with the base case given by Theorem \ref{thm1.1}. Following the same inductive approach, and using Theorem \ref{thm5.1} as the base case, we may obtain the following result concerning expansions of vertex-disjoint unions of paths, cycles and stars.

\begin{thm}\label{thm5.2}
Let integers $r\geq3$, $k\geq1$, $p\geq0$, $\ell_1,\dots,\ell_k\geq 5$, $t_1,\dots,t_{p}\geq2$ if $p\geq1$, and $t=\sum_{i=1}^k\lfloor\frac{\ell_i+1}{2}\rfloor+p-1$. Suppose $Q_i\in\{P_{\ell_i}^r,C_{\ell_i}^r\}$ and $n$ is sufficiently large.

\textbf{(i)} If $s \leq t+r-1$, then
\begin{eqnarray*}
{\rm{ex}}_r(n,K_s^{(r)},Q_1\cup\cdots\cup Q_k\cup S^r_{t_1}\cup\cdots\cup S^r_{t_p})= (1+o(1))\cN(K_s^{(r)},S_{n,t}^r)=\binom{t}{s-r+1}\frac{n^{r-1}}{(r-1)!}+o(n^{r-1}).
\end{eqnarray*}

\textbf{(ii)} If $s \geq t+r$, then
\begin{eqnarray*}
{\rm{ex}}_r(n,K_s^{(r)},Q_1\cup\cdots\cup Q_k\cup S^r_{t_1}\cup\cdots\cup S^r_{t_p})= o(n^{r-1}).
\end{eqnarray*}
\end{thm}

\smallskip
\noindent $\bullet$  In Theorems \ref{thm1.1} and \ref{thm3}, we establish only the asymptotics of the corresponding generalized Tur\'{a}n numbers,  since our arguments rely on an inductive argument whose base case is an asymptotic result of Axenovich, Gerbner, Liu, and Patkós \cite{AGLP} (see Theorem \ref{Gerbner-path}).
Consequently, if the exact generalized Tur\'{a}n numbers for $P_\ell^r$ and $C_\ell^r$ were known, our method may yield the exact generalized Tur\'{a}n numbers for $Q_1\cup\cdots\cup Q_k$ and $Q_1\cup\cdots\cup Q_k\cup S^r_{t_1}\cup\cdots\cup S^r_{t_p}$, where $Q_i\in\{P_{\ell_i}^r,C_{\ell_i}^r\}$.

\subsection{\normalsize A general bound on $\ex_r(n,K_s^{(r)},\cF)$}
In Section \ref{3}, we establish a general upper bound on ${\rm{ex}}_r(n,K_s^{(r)},F^r)$ for an arbitrary graph $F$ (see Lemma \ref{lem0}). Theorems \ref{Pikhurko} and \ref{thm0} show that this bound is sharp in a specific case.
Next, we present a more general bound on $\ex_r(n,K_s^{(r)},\cF)$ for any $r$-graph $\cF$. Given a graph $F$, an $r$-graph $\cH$ is a \textit{Berge-$F$} if there is a bijection $\phi: E(F)\rightarrow E(\cH)$ such that $e\subseteq \phi(e)$ for each $e\in E(F)$. Berge \cite{A2} defined the Berge-cycle, and Gy\H{o}ri, Katona and Lemons \cite{B6} defined the Berge-path. Later, Gerbner and Palmer \cite{B3} extended the definition to arbitrary graphs $F$. A brief survey of Tur\'{a}n-type results for Berge-graphs can be found in Subsection 5.2.2 in \cite{A666}. 
Subsequently, Balko et al. \cite{BGKKP} further generalized the definition of Berge-graphs to arbitrary hypergraphs. For $s\geq r$ and an $r$-graph $\cF$, an $s$-graph $\cG$ is a \textit{Berge-$\cF$} if there is a bijection $\phi: E(\cF)\rightarrow E(\cG)$ such that $e\subseteq \phi(e)$ for each $e\in E(\cF)$. Note that for a fixed $r$-graph $\cF$, there are many $s$-graphs that form a Berge-$\cF$. For convenience, we refer to this collection of $s$-graphs as ``Berge-$\cF$''.
The maximum number of hyperedges in an $n$-vertex $s$-graph containing no subhypergraph isomorphic to any Berge-$\cF$ is denoted by ${\rm{ex}}_s(n,{\rm Berge}{\text -}\cF)$.

The problem of counting copies of the complete $r$-graph $K_s^{(r)}$ in an $n$-vertex $\cF$-free $r$-graph is closely related to the study of Berge-hypergraphs. 
An explicit relationship between these two classes of problems was first established by Balko et al. \cite{BGKKP}, and we provide an alternative proof here.

\begin{prop}[Balko, Gerbner, Kang, Kim and Palmer \cite{BGKKP}]\label{prop6.3}
Let $s\geq r\geq 2$ be integers and $\cF$ be an $r$-graph. Then
\begin{eqnarray*}
{\rm{ex}}_s(n,{\rm Berge}{\text -}\cF)-\ex_r(n,\cF)\leq \ex_r(n,K_s^{(r)},\cF)\leq {\rm{ex}}_s(n,{\rm Berge}{\text -}\cF).
\end{eqnarray*}
\end{prop}
\begin{proof}[{\bf{Proof.}}]
Let $\cH$ be an $n$-vertex $\cF$-free $r$-graph. It is easy to see that the $s$-clique hypergraph of $\cH$ must not contain a Berge-$\cF$. This implies that $\ex_r(n,K_s^{(r)},\cF)\leq {\rm{ex}}_s(n,{\rm Berge}{\text -}\cF)$.

Now we consider the lower bound. Let $\cG$ be an $n$-vertex $s$-graph containing no subhypergraph isomorphic to any Berge-$\cF$, and let $E(\cG)=\{e_1,e_2,\dots,e_x\}$ denote its hyperedge set. We construct an $r$-graph $\cH$ on the vertex set $V(\cG)$ as follows. At step $i$ ($1\leq i\leq x$), we attempt to choose an $r$-subset of $e_i$ that has not yet been selected as a hyperedge of $\cH$. If no such $r$-subset exists, that is, if every $r$-subset of $e_i$ is already a hyperedge of $\cH$, then no hyperedge is added at this step.

Any copy of $\cF$ in $\cH$ would correspond to a Berge-$\cF$ in $\cG$. Therefore, $\cH$ is $\cF$-free. By the definition of $\cH$, for each hyperedge $e_i\in E(\cG)$ for which no $r$-subset was added to $\cH$, the vertices of $e_i$ span a copy of $K_s^{(r)}$ in $\cH$. Consequently, the number of such hyperedges of $\cG$ is at most the number of copies of $K_s^{(r)}$ in $\cH$. Hence, each hyperedge $e_i\in E(\cG)$ corresponds either to a hyperedge of $\cH$ or to a copy of $K_s^{(r)}$ spanned by its vertices in $\cH$. This implies that $e(\cG)\leq \ex_r(n,K_s^{(r)},\cF)+\ex_r(n,\cF)$. Thus, ${\rm{ex}}_s(n,{\rm Berge}{\text -}\cF)-\ex_r(n,\cF)\leq \ex_r(n,K_s^{(r)},\cF)$.
\end{proof}

Proposition \ref{prop6.3} immediately yields the following corollary, which was previously established by Gerbner and Palmer \cite{GePa2}.

\begin{cor}[Gerbner and Palmer \cite{GePa2}]
Let $s\geq2$ be an integer and $F$ be a graph. Then
\begin{eqnarray*}
\ex(n,K_s,F)\leq {\rm{ex}}_s(n,{\rm Berge}{\text -}F)\leq \ex(n,K_s,F)+\ex(n,F).
\end{eqnarray*}
\end{cor}


\bigskip
\noindent{\bf{Funding}}
\smallskip


The research of Zhou and Yuan was supported by the National Natural Science Foundation of China (Nos.~12271337 and 12371347).

The research of Zhao was supported by the China Scholarship Council (No.~202506210250). 

\bigskip
\noindent{\bf{Declaration of interest}}
\smallskip

The authors declare no known conflicts of interest.

\bigskip
\noindent{\bf{Acknowledgements}}

We would like to express our sincere gratitude to Professor D\'{a}niel Gerbner for his helpful comments on an earlier draft of this paper.

\appendix
\section{\normalsize Proof of Theorem \ref{Zhou-Yuan}}
First, we state the following summary result regarding expansions of paths and cycles, which will be used in the proof.

\begin{thm}[F\"{u}redi, Jiang, Seiver \cite{FuJS}, Kostochka, Mubayi, Verstra\"{e}te \cite{KMV}, Frankl and F\"{u}redi \cite{FrFu}]\label{lem3.1}
Let $r\geq3$ and $\ell\geq3$. For sufficiently large $n$,
\begin{eqnarray*}
{\rm{ex}}_r(n,P_\ell^r)=\binom{n}{r}-\binom{n-\lfloor\frac{\ell-1}{2}\rfloor}{r}+
\begin{cases}
0 &{\rm{if}}\ \ell\ {\rm{is}}\ {\rm{odd}}, \\
\binom{n-\lfloor\frac{\ell-1}{2}\rfloor-2}{r-2} &{\rm{if}}\ \ell\ {\rm{is}}\ {\rm{even}}.
\end{cases}
\end{eqnarray*}
The same result holds for $C_\ell^r$ when $r\geq3$ and $\ell\geq4$, except the case $(r,\ell)=(3,4)$, in which case
\begin{eqnarray*}
{\rm{ex}}_3(n,C_4^3)=\binom{n}{3}-\binom{n-1}{3}+\max\left\{ n - 3, 4 \left\lfloor \frac{n - 1}{4} \right\rfloor \right\}.
\end{eqnarray*}
\end{thm}

Observe that ${\rm{ex}}_r(n,P^r_{3})={\rm{ex}}_r(n,K^r_{3})={\rm{ex}}_r(n,C^r_{3})$ for sufficiently large $n$. Now let us start with the proof of Theorem \ref{Zhou-Yuan} that we restate here for convenience.

\begin{thm}[Zhou and Yuan \cite{ZhY2}]
Let integers $r\geq3$, $k\geq1$, $0\leq p\leq k$ and $\ell_1,\dots,\ell_k\geq3$, where $\ell_i\neq4$ for all $i$ if $r=3$. Then for sufficiently large $n$,
\begin{eqnarray*}
{\rm{ex}}_r(n,P_{\ell_1}^r\cup\cdots\cup P_{\ell_{p}}^r\cup C_{\ell_{p+1}}^r\cup\cdots\cup C_{\ell_{k}}^r)= \binom{n}{r}-\binom{n-\sum_{i=1}^{k}\big\lfloor\frac{\ell_i+1}{2}\big\rfloor+1}{r}+c(n;\ell_1,\dots,\ell_k),
\end{eqnarray*}
where $c(n;\ell_1,\dots,\ell_k)=\binom{n-\sum_{i=1}^{k}\lfloor\frac{\ell_i+1}{2}\rfloor-1}{r-2}$ if all of $\ell_1,\dots,\ell_k$ are even, and $c(n;\ell_1,\dots,\ell_k)=0$ otherwise.
\end{thm}

\begin{proof}[{\bf{Proof.}}]
For the lower bound, we consider the $r$-graph on $n$ vertices consisting of all hyperedges intersecting a fixed set $U$ of $\sum_{i=1}^{k}\big\lfloor\frac{\ell_i+1}{2}\big\rfloor-1$ vertices, along with all hyperedges disjoint from $U$ containing two fixed vertices outside $U$ when all of $\ell_1,\dots,\ell_k$ are even. Observe that any $P_{\ell_i}^r$ or $C_{\ell_i}^r$ in the resulting $r$-graph contains at least $\big\lfloor\frac{\ell_i+1}{2}\big\rfloor$ vertices in $U$. Thus, the resulting $r$-graph is $\mathcal{G}_1$-free as $|U|<\sum_{i=1}^{k}\big\lfloor\frac{\ell_i+1}{2}\big\rfloor$. Clearly, the number of hyperedges of the resulting $r$-graph is
$\binom{n}{r}-\binom{n-\sum_{i=1}^{k}\lfloor\frac{\ell_i+1}{2}\rfloor+1}{r}+c(n;\ell_1,\dots,\ell_k)$.

We now proceed with the upper bound. Let $\cH$ be an $r$-graph on $n$ vertices with
\begin{eqnarray*}
e(\cH)> \binom{n}{r}-\binom{n-\sum_{i=1}^{k}\big\lfloor\frac{\ell_i+1}{2}\big\rfloor+1}{r}+c(n;\ell_1,\dots,\ell_k).
\end{eqnarray*}
We will show that $\cH$ contains a copy of $P_{\ell_1}^r\cup\cdots\cup P_{\ell_{p}}^r\cup C_{\ell_{p+1}}^r\cup\cdots\cup C_{\ell_{k}}^r$ by induction on $k$. If $k=1$, then $e(\cH)>{\rm{ex}}_r(n,P^r_{\ell_1})={\rm{ex}}_r(n,C^r_{\ell_1})$ by Theorem \ref{lem3.1}, and hence the conclusion holds. Suppose $k\geq2$ and the statement holds for any $k'<k$. Since it is possible that $p=0$, for convenience we denote $P_{\ell_1}^r\cup\cdots\cup P_{\ell_{p}}^r\cup C_{\ell_{p+1}}^r\cup\cdots\cup C_{\ell_{k}}^r$ by $F_{\ell_1}^r\cup\cdots\cup F_{\ell_k}^r$.

Since $n$ is sufficiently large, we have
\begin{eqnarray*}
e(\cH)> \binom{n}{r}-\binom{n-\sum_{i=1}^{k}\big\lfloor\frac{\ell_i+1}{2}\big\rfloor+1}{r}>\binom{n}{r}-\binom{n-\big\lfloor\frac{\ell_1+1}{2}\big\rfloor+1}{r}+\binom{n-\big\lfloor\frac{\ell_1+1}{2}\big\rfloor-1}{r-2},
\end{eqnarray*}
which implies $\cH$ contains a copy of $F_{\ell_1}^r$ by Theorem \ref{lem3.1}, denoted by $F$. Therefore, $\ell_1(r-1)\leq |V(F)|\leq \ell_1(r-1)+1$.
We may assume that
\begin{eqnarray*}
e(\cH-V(F))\leq \binom{n-|V(F)|}{r}-\binom{n-|V(F)|-\sum_{i=2}^{k}\big\lfloor\frac{\ell_i+1}{2}\big\rfloor+1}{r}+c(n-|V(F)|;\ell_2,\dots,\ell_k).
\end{eqnarray*}
Otherwise, by induction hypothesis, we may find a copy of $F_{\ell_2}^r\cup\cdots\cup F_{\ell_k}^r$ in $\cH-V(F)$. Thus, $H$ contains a copy of $F_{\ell_1}^r\cup\cdots\cup F_{\ell_k}^r$ and we are done.

Let $E_{V(F)}$ denote the set of hyperedges of $\cH$ incident to vertices in $V(F)$. Then
\begin{eqnarray}
|E_{V(F)}|&=& e(\cH)-e(\cH-V(F)) \notag \\
&>& \binom{n}{r}-\binom{n-\sum_{i=1}^{k}\big\lfloor\frac{\ell_i+1}{2}\big\rfloor+1}{r}+c(n;\ell_1,\dots,\ell_k) \notag \\
&{}& -\bigg[\binom{n-|V(F)|}{r}-\binom{n-|V(F)|-\sum_{i=2}^{k}\big\lfloor\frac{\ell_i+1}{2}\big\rfloor+1}{r}+c(n-|V(F)|;\ell_2,\dots,\ell_k)\bigg] \notag \\
&=& \frac{\lfloor\frac{\ell_1+1}{2}\big\rfloor}{(r-1)!}n^{r-1}+o(n^{r-1}).
\end{eqnarray}

Let $E_{V(F)}^{1}=\big\{e\in E(\cH)\,\big|\,|e\cap V(F)|=1\big\}$. For every $(r-1)$-set $S$ in $V(\cH)\backslash V(F)$, set $E_{V(F)}^{1}(S)=\big\{e\in E_{V(F)}^{1}\,\big|\,S\subseteq e\big\}$. Note that $0\leq|E_{V(F)}^{1}(S)|\leq |V(F)|$ for any $S$. Let
$$X=\Big\{S\subseteq V(\cH)\backslash V(F)\,\big|\,|E_{V(F)}^{1}(S)|\geq \Big\lfloor\frac{\ell_1+1}{2}\Big\rfloor\Big\},$$
$$Y=\Big\{S\subseteq V(\cH)\backslash V(F)\,\big|\,E_{V(F)}^{1}(S)\notin X\Big\}.$$
Note that the number of hyperedges in $E_{V(F)}$ containing at least two vertices from $V(F)$ is at most $\binom{|V(F)|}{2}\binom{n-2}{r-2}$. Then
\begin{eqnarray}
|E_{V(F)}|&\leq& |X|\cdot |V(F)|+|Y|\cdot \Big(\Big\lfloor\frac{\ell_1+1}{2}\Big\rfloor-1\Big)+\binom{|V(F)|}{2}\binom{n-2}{r-2} \notag \\
&\leq& |X|\cdot |V(F)|+\binom{n-|V(F)|}{r-1}\Big(\Big\lfloor\frac{\ell_1+1}{2}\Big\rfloor-1\Big)+\binom{|V(F)|}{2}\binom{n-2}{r-2}.
\end{eqnarray}
Combining (A.1) and (A.2), we get
\begin{eqnarray}
|X|\geq \frac{n^{r-1}}{|V(F)|\cdot(r-1)!}+o(n^{r-1}).
\end{eqnarray}

By the definition of $X$, for any $S\in X$ there is at least $\big\lfloor\frac{\ell_1+1}{2}\big\rfloor$ vertices in $V(F)$ such that it forms a hyperedge in $E_{V(F)}^{1}(S)$ together with $S$. Therefore, by (A.3) and pigeonhole principle, there exists a $\big\lfloor\frac{\ell_1+1}{2}\big\rfloor$-set $U$ in $V(F)$ such that, there are at least
\begin{eqnarray*}
\frac{n^{r-1}}{|V(F)|\cdot(r-1)!}\cdot \frac{1}{\binom{|V(F)|}{\lfloor\frac{\ell_1+1}{2}\rfloor}}+o(n^{r-1})\geq \frac{n^{r-1}}{(\ell_1(r-1)+1)\cdot(r-1)!}\cdot \frac{1}{\binom{(\ell_1(r-1)+1)}{\lfloor\frac{\ell_1+1}{2}\rfloor}}+o(n^{r-1})
\end{eqnarray*}
$(r-1)$-sets in $V(\cH)\backslash V(F)$ such that every $(r-1)$-set forms a hyperedge with each vertex in $U$. Let $\mathcal{S}_U$ be the collection of all such $(r-1)$-sets. Then
\begin{eqnarray}
|\mathcal{S}_U|=\Omega(n^{r-1}).
\end{eqnarray}

Let $E_{U}$ denote the set of hyperedges of $\cH$ incident to vertices in $U$. Then
\begin{eqnarray*}
|E_{U}|&\leq& \binom{|U|}{r}+\binom{|U|}{r-1}\binom{n-|U|}{1}+\cdots+\binom{|U|}{1}\binom{n-|U|}{r-1} \\
&=& \binom{n}{r}-\binom{n-|U|}{r}=\binom{n}{r}-\binom{n-\big\lfloor\frac{\ell_1+1}{2}\big\rfloor}{r},
\end{eqnarray*}
where the second equality holds by Vandermonde's identity.

Now we compare the values of $c(n;\ell_1,\dots,\ell_k)$ and $c(n-|U|;\ell_2,\dots,\ell_k)$.
\begin{itemize}
    \item If all of $\ell_1,\dots,\ell_k$ are even, then $c(n;\ell_1,\dots,\ell_k)=c(n-|U|;\ell_2,\dots,\ell_k)=\binom{n-\sum_{i=1}^{k}\big\lfloor\frac{\ell_i+1}{2}\big\rfloor-1}{r-2}$.
    \item If at least two of $\ell_1,\dots,\ell_k$ are odd, then $c(n;\ell_1,\dots,\ell_k)=c(n-|U|;\ell_2,\dots,\ell_k)=0$.
    \item If exactly one of $\ell_1,\dots,\ell_k$ is odd and $\ell_1$ is even, then $c(n;\ell_1,\dots,\ell_k)=c(n-|U|;\ell_2,\dots,\ell_k)=0$.
    \item If exactly one of $\ell_1,\dots,\ell_k$ is odd and $\ell_1$ is odd, then we exchange $F^r_{\ell_1}$ and $F^r_{\ell_2}$, and apply arguments similar to those above to $F^r_{\ell'_1}\cup F^r_{\ell'_2}\cup\cdots\cup F^r_{\ell_k}$, where $F^r_{\ell'_1}=F^r_{\ell_2}$ and $F^r_{\ell'_2}=F^r_{\ell_1}$. Therefore, one of $\ell'_1,\ell'_2,\dots,\ell_k$ is odd and $\ell'_1$ is even. It follows that $c(n;\ell'_1,\ell'_2,\dots,\ell_k)=c(n-|U|;\ell'_2,\dots,\ell_k)=0$.
\end{itemize}
Then $c(n;\ell_1,\dots,\ell_k)-c(n-|U|;\ell_2,\dots,\ell_k)\geq0$ always holds. Therefore,
\begin{eqnarray*}
&{}& \binom{n}{r}-\binom{n-\sum_{i=1}^{k}\big\lfloor\frac{\ell_i+1}{2}\big\rfloor+1}{r}+c(n;\ell_1,\dots,\ell_k) \\
&{}&-\bigg[\binom{n-\big\lfloor\frac{\ell_1+1}{2}\big\rfloor}{r}-\binom{n-\big\lfloor\frac{\ell_1+1}{2}\big\rfloor-\sum_{i=2}^{k}\big\lfloor\frac{\ell_i+1}{2}\big\rfloor+1}{r}+c(n-|U|;\ell_2,\dots,\ell_k)\bigg] \\
&=& \binom{n}{r}-\binom{n-\big\lfloor\frac{\ell_1+1}{2}\big\rfloor}{r}+c(n;\ell_1,\dots,\ell_k)-c(n-|U|;\ell_2,\dots,\ell_k) \\
&\geq& |E_{U}|.
\end{eqnarray*}
Namely,
\begin{eqnarray*}
e(\cH-U)&=& e(\cH)-|E_{U}| \\
&>& \binom{n}{r}-\binom{n-\sum_{i=1}^{k}\big\lfloor\frac{\ell_i+1}{2}\big\rfloor+1}{r}+c(n;\ell_1,\dots,\ell_k)-|E_{U}| \\
&\geq& \binom{n-\big\lfloor\frac{\ell_1+1}{2}\big\rfloor}{r}-\binom{n-\big\lfloor\frac{\ell_1+1}{2}\big\rfloor-\sum_{i=2}^{k}\big\lfloor\frac{\ell_i+1}{2}\big\rfloor+1}{r}+c(n-|U|;\ell_2,\dots,\ell_k),
\end{eqnarray*}
which implies $\cH-U$ contains a copy of $F^r_{\ell_2}\cup\cdots\cup F^r_{\ell_k}$ by induction hypothesis, denoted by $F'$. Therefore, $V(F')\subseteq V(\cH)\backslash U$ and $\sum_{i=2}^k(r-1)\ell_i\leq |V(F')|\leq \sum_{i=2}^k[(r-1)\ell_i+1]$.

Next, we will find a copy of $F^r_{\ell_1}$ that is vertex-disjoint from $V(F')$ in $\cH$. Therefore, $F^r_{\ell_1}$ and $F'$ together form a copy of $F^r_{\ell_1}\cup\cdots\cup F^r_{\ell_k}$ in $\cH$. Consider the $(r-1)$-graph $\cH^{r-1}=(V(\cH)\backslash U,\mathcal{S}_U)$. Note that the number of hyperedges in $E(\cH^{r-1})$ that contain at least one vertex from $V(F')$ is at most $|V(F')|\cdot\binom{n-|U|-1}{r-2}\leq \sum_{i=2}^k[(r-1)\ell_i+1]\cdot\binom{n-|U|-1}{r-2}$. By (A.4) and Theorem \ref{lem3.2}, we have
\begin{eqnarray*}
e(\cH^{r-1}[V(\cH)\backslash(U\cup V(F'))])&\geq& e(\cH^{r-1})-\sum_{i=2}^k\big[(r-1)\ell_i+1\big]\cdot\binom{n-|U|-1}{r-2} \\
&=& |\mathcal{S}_U|-\sum_{i=2}^k\big[(r-1)\ell_i+1\big]\cdot\binom{n-|U|-1}{r-2} \\
&=& \Omega(n^{r-1}) \\
&>& {\rm{ex}}_{r-1}(n-|U|,S_2^{r-1}),
\end{eqnarray*}
and thus $\cH^{r-1}[V(\cH)\backslash(U\cup V(F'))]$ contains a copy of $S_2^{r-1}$. Note that $|V(S_2^{r-1})|=2r-3$. For any fixed large constant $2\leq N=o(n)$, by (A.4) and Theorem \ref{lem3.2}, we have
\begin{eqnarray*}
&{}& e(\cH^{r-1})-\sum_{i=2}^k\big((r-1)\ell_i+1\big)\cdot\binom{n-|U|-1}{r-2}-(N-1)(2r-3)\binom{n-|U|-1}{r-2} \\
&\geq& |\mathcal{S}_U|-\sum_{i=2}^k\big((r-1)\ell_i+1\big)\cdot\binom{n-|U|-1}{r-2}-o(n)\cdot(2r-3)\binom{n-|U|-1}{r-2} \\
&=& \Omega(n^{r-1}) \\
&>& {\rm{ex}}_{r-1}(n-|U|,S_2^{r-1}).
\end{eqnarray*}
This implies that $\cH^{r-1}[V(\cH)\backslash(U\cup V(F'))]$ contains a copy of $NS_2^{r-1}$. We denote it and its hyperedges by $S^{(1)}\cup\cdots\cup S^{(N)}$ and $e^{(1)}_1,e^{(1)}_2,\dots,e^{(N)}_1,e^{(N)}_2$, where $e^{(i)}_1,e^{(i)}_2\in E(S^{(i)})$.

Let $U=\{v_1,\dots,v_{\lfloor\frac{\ell_1+1}{2}\rfloor}\}$. Then we may construct a copy of $P^r_{2\lfloor\frac{\ell_1+1}{2}\rfloor}$ and a copy of $C^r_{2\lfloor\frac{\ell_1+1}{2}\rfloor}$ that are vertex-disjoint from $V(F')$ as follows.
$$P^r_{2\lfloor\frac{\ell_1+1}{2}\rfloor}:\ \{v_1\}\cup e_2^{(1)},\{v_1\}\cup e_1^{(2)},\{v_2\}\cup e_2^{(2)},\dots,\{v_{\lfloor\frac{\ell_1+1}{2}\rfloor}\}\cup e_2^{({\lfloor\frac{\ell_1+1}{2}\rfloor})},\{v_{\lfloor\frac{\ell_1+1}{2}\rfloor}\}\cup e_1^{({\lfloor\frac{\ell_1+1}{2}\rfloor}+1)},$$
$$C^r_{2\lfloor\frac{\ell_1+1}{2}\rfloor}:\ \{v_1\}\cup e_1^{(1)},\{v_2\}\cup e_2^{(1)},\{v_2\}\cup e_1^{(2)},\dots,\{v_{\lfloor\frac{\ell_1+1}{2}\rfloor}\}\cup e_1^{({\lfloor\frac{\ell_1+1}{2}\rfloor})},\{v_{1}\}\cup e_2^{({\lfloor\frac{\ell_1+1}{2}\rfloor})}.$$

Observe that $P^r_{\ell_1}\subseteq P^r_{2\lfloor\frac{\ell_1+1}{2}\rfloor}$, and $C^r_{2\lfloor\frac{\ell_1+1}{2}\rfloor}=C^r_{\ell_1}$ when $\ell_1$ is even. Now we construct a $C^r_{\ell_1}$ in the case when $\ell_1$ is odd. Similarly, by (A.4) and Theorem \ref{lem3.1}, we have
\begin{eqnarray*}
&{}& e(\cH^{r-1})-\sum_{i=2}^k\big((r-1)\ell_i+1\big)\cdot\binom{n-|U|-1}{r-2}-\ell_1(2r-3)\binom{n-|U|-1}{r-2} \\
&=& \Omega(n^{r-1}) \\
&>& {\rm{ex}}_{r-1}(n-|U|,P_3^{r-1}).
\end{eqnarray*}
This implies that $\cH^{r-1}[V(\cH)\backslash(U\cup V(F'))]$ contains a copy of $\ell_1S_2^{r-1}\cup P_3^{r-1}$. We take a copy of $\frac{\ell_1-3}{2}S_2^{r-1}\cup P_3^{r-1}$, and denote it and its hyperedges by $S^{(1)}\cup\cdots\cup S^{(\frac{\ell_1-3}{2})}\cup P$ and $e^{(1)}_1,e^{(1)}_2,\dots,e^{(\frac{\ell_1-3}{2})}_1,e^{(\frac{\ell_1-3}{2})}_2,e_1,e_2,e_3$, where $e^{(i)}_1,e^{(i)}_2\in E(S^{(i)})$ and $e_1,e_2,e_3\in E(P)$. Then we may construct a copy of $C^r_{\ell_1}$ that are vertex-disjoint from $V(F')$ as follows.
$$C^r_{\ell_1}:\ \{v_1\}\cup e_1^{(1)},\{v_2\}\cup e_2^{(1)},\dots,\{v_{\frac{\ell_1-1}{2}}\}\cup e_2^{({\frac{\ell_1-3}{2}})},\{v_{\frac{\ell_1-1}{2}}\}\cup e_1,\{v_{\frac{\ell_1+1}{2}}\}\cup e_2,\{v_{1}\}\cup e_3.$$

Thus, we may always find a copy of $F^r_{\ell_1}$ that is vertex-disjoint from $V(F')$ in $\cH$. This completes the proof.
\end{proof}

\end{document}